\documentclass[a4paper,twopage,reqno,11pt]{amsart}
\usepackage{amsfonts, amsbsy, amsmath,  amssymb, latexsym}
\usepackage{mathrsfs}

\usepackage{pdfsync}
\usepackage[top=30mm,right=30mm,bottom=30mm,left=30mm]{geometry}

\usepackage{stmaryrd}
\usepackage{bm}

\usepackage{fancyhdr}

\makeatletter
\newcommand{\imod}[1]{\allowbreak\mkern4mu({\operator@font mod}\,\,#1)}
\makeatother

\renewcommand{\a}{\alpha}
\renewcommand{\b}{\beta}
\newcommand{\F}{\mathbb{F}_{q}}
 \newcommand{\e}{\epsilon}
 
 \renewcommand{\O}{\Omega}
 
 \renewcommand{\to}{\rightarrow}
 \newcommand{\s}{\sigma}
 
 \newcommand{\C}{\mathcal{C}}
\newcommand{\G}{\bar{G}}

\newcommand{\M}{\bar{M}}

\newcommand{\leqs}{\leqslant}
\newcommand{\geqs}{\geqslant}
 
\newcommand{\what}{\widehat} 
 \newcommand{\vs}{\vspace{2mm}}

\newtheorem{theorem}{Theorem}

\newtheorem{thm}{Theorem}[section]
\newtheorem{prop}[thm]{Proposition}
\newtheorem{lem}[thm]{Lemma}
\newtheorem{cor}[thm]{Corollary}

\theoremstyle{definition}
\newtheorem{definition}[theorem]{Definition}
\newtheorem{remark}[theorem]{Remark}
\newtheorem{defn}[thm]{Definition}
\newtheorem{remk}[thm]{Remark}

\begin{document}
 
 \author{S. Hassan Alavi}
 \address{S.H. Alavi, Department of Mathematics, Faculty of Science, Bu-Ali Sina University, Hamedan, Iran}
 \email{alavi.s.hassan@gmail.com}
  
 \author{Timothy C. Burness}
 %\thanks{Corresponding author: Dr. T.C. Burness}
 \address{T.C. Burness, School of Mathematics, University of Bristol, Bristol BS8 1TW, UK}
 \email{t.burness@bristol.ac.uk}
 
 \title{Large subgroups of simple groups}
 
 
 
\begin{abstract}
Let $G$ be a finite group. A proper subgroup $H$ of $G$ is said to be large if the order of $H$ satisfies the bound $|H|^3 \geqs |G|$. In this note we determine all the large maximal subgroups of finite simple groups, and we establish an analogous result for simple algebraic groups  
(in this context, largeness is defined in terms of dimension). An application to triple factorisations of simple groups (both finite and algebraic) is discussed.
\end{abstract}

\thanks{Burness thanks the Centre for the Mathematics of Symmetry and Computation at the University of Western Australia for their generous hospitality, and both authors thank Professors Bill Kantor, Martin Liebeck and Cheryl Praeger for helpful comments.}

\date{\today}
\maketitle

%\setcounter{tocdepth}{1}
%\tableofcontents

\begin{center}
{\em Dedicated to the  memory of  \'Akos Seress.}
\end{center}

\vspace{5mm}

\section{Introduction}\label{s:intro}

Let $G$ be a group. A \emph{triple factorisation} of $G$ is a factorisation of the form $G=ABA$, where $A$ and $B$ are proper subgroups of $G$. Such factorisations arise naturally in several different contexts. For example, the Bruhat decomposition $G=BNB$ of a group of Lie type, where $B$ is a Borel subgroup and $N$ is the normaliser of a maximal torus, plays an important role in the study of such groups. In a different direction, a triple factorisation corresponds to a flag-transitive point-line incidence geometry in which each pair of points is incident with at least one line (see \cite[Lemma 3]{HigmanM}). 

Determining the triple factorisations of a given group is a difficult problem. For a finite group $G$, a starting point is the easy observation that $G=ABA$ only if 
\begin{equation}\label{e:max}
\max\{|A|^3,|B|^3\} \geqs |G|.
\end{equation} 
This motivates the following definition.

\begin{definition}\label{d:large}
Let $G$ be a finite group. A proper subgroup $H$ of $G$ is said to be \emph{large} if the order of $H$ satisfies the bound $|H|^3 \geqs |G|$.
\end{definition}

In \cite{AP}, Alavi and Praeger develop a general framework for studying triple factorisations of finite groups in terms of group actions. In particular,  a reduction strategy for classifying triple factorisations is presented in \cite[Section 1.1]{AP}, in which it is reasonable to assume that the subgroup $A$ appearing in a factorisation $G=ABA$ is maximal (and also core-free). Simultaneous triple factorisations of the form $G=ABA = BAB$ are particularly interesting from a geometric point of view, and in this situation it is reasonable to assume that both $A$ and $B$ are maximal subgroups of $G$. Factorisations of this form have been studied in two recent papers \cite{ABP, GJ}. In \cite{ABP}, $G={\rm GL}(V)$ and the subgroups $A$ and $B$ either stabilise a subspace of $V$, or stabilise a decomposition $V=V_1 \oplus V_2$ with $\dim V_1 = \dim V_2$. In  \cite{GJ}, the case where $G=S_n$ and $A,B$ are maximal conjugate subgroups is investigated.
 
The main aim of this paper is to determine the large maximal subgroups of finite simple groups, which can be viewed as a first step towards a general investigation of triple factorisations of simple groups. By the Classification of Finite Simple Groups, every nonabelian finite simple group is isomorphic to an alternating group $A_n$ of degree $n \geqs 5$, a group of Lie type defined over a finite field $\F$ (of classical or exceptional type), or one of $26$ sporadic groups. There is a vast literature on the subgroup structure of finite simple groups, and their maximal subgroups in particular. 

The problem of determining the ``large" maximal subgroups of finite simple groups has a long history, with many applications. For alternating groups, it is closely related to the following old question in permutation group theory (see \cite{Maroti} and the references therein): how large can a primitive group $G$ of degree $n$ be, assuming $G$ does not contain $A_n$? For groups of Lie type, some related results are established in \cite{L} and \cite{LieSax}. The main theorem of \cite{L} describes the maximal subgroups $H$ of a simple classical group with natural (projective) module of dimension $n$ over $\mathbb{F}_q$ that satisfy the bound $|H| \geqs q^{3n}$ (a special case of this result, Theorem \ref{lieb}, plays a key role in our analysis of almost simple irreducible subgroups of classical groups). In the same paper, the largest irreducible proper subgroups of simple classical groups are also determined (see \cite[Section 5]{L}). For a group $G$ of exceptional Lie type over $\mathbb{F}_q$, the main theorem of \cite{LieSax} records the maximal subgroups $H$ of $G$ such that $|H| \geqs q^{k(G)}$, where $k(G)$ is an integer specified in \cite[Table 1]{LieSax}. In particular, one can read off the subgroups with $|H|^2 \geqs |G|$. This theorem is used in the classification of the finite primitive permutation groups of rank three \cite{LSr3}, and a related result, \cite[Theorem 1.2]{LSh4}, plays an important role in the proof of Dixon's conjecture on the probabilistic generation of finite simple groups.

Let us now state our main results. Throughout this paper we adopt the standard notation of Kleidman and Liebeck \cite{KL}. Note that in Theorems \ref{t:main1} -- \ref{t:main4} below, we are stating that $H$ is large if and only if $H$ is isomorphic to one of the subgroups in the relevant lists.
%; in particular, the difficult problem of determining the conjugacy classes of maximal subgroups of simple groups does not play a role in this paper.

\begin{theorem}\label{t:main1}
Let $G=A_n$ be an alternating group of degree $n \geqs 5$, and let $H$ be a maximal subgroup 
of $G$. Then $H$ is large if and only if $H$ is either intransitive or imprimitive on $\{1,\ldots, n\}$, or if $(n,H)$ is one of the following:
$$\begin{array}{lllll}
(5,D_{10}), & (6, {\rm L}_{2}(5)), & (7, {\rm L}_{2}(7)), & (8, {\rm AGL}_{3}(2)), & (9,3^2.{\rm SL}_{2}(3)), \\
 (9,{\rm P\Gamma L}_{2}(8)), & (10, {\rm M}_{10}), & (11, {\rm M}_{11}), & (12, {\rm M}_{12}), & (13,{\rm L}_{3}(3)), \\
 (15,A_8), & (16,{\rm AGL}_{4}(2)), & (24, {\rm M}_{24}). & & 
\end{array}$$
\end{theorem}

\begin{theorem}\label{t:main2}
Let $G$ be a sporadic simple group, and let $H$ be a maximal subgroup 
of $G$. 
\begin{itemize}\addtolength{\itemsep}{0.2\baselineskip}
\item[{\rm (i)}] If $G \in \{{\rm J}_{4}, {\rm Co}_{1}, {\rm Fi}_{24}', {\rm Th}, {\rm Ru}, \mathbb{B}, \mathbb{M}\}$ then $H$ is large if and only if $(G,H)$ is one of the cases listed in Table \ref{tab:spor}.
\item[{\rm (ii)}] In the remaining cases, $H$ is large unless $(G,H)$ is one of the following:
$$\begin{array}{lllll}
({\rm M}_{24}, {\rm L}_{2}(7)), & ({\rm J}_{1}, 7{:}6), & ({\rm J}_{2}, A_5), & ({\rm Co}_{2}, 5^{1+2}{:}4S_4), &  ({\rm Co}_{3}, A_4 \times S_5), \\
({\rm Fi}_{23}, {\rm L}_{2}(23)), & ({\rm Suz},A_7), & ({\rm He}, 5^2{:}4A_4), & ({\rm HN}, 3^{1+4}{:}4.A_5), & ({\rm O'N},A_7), \\
({\rm Ly}, 67{:}22), & ({\rm Ly}, 37{:}18). & & & 
\end{array}$$
\end{itemize}
\end{theorem}

\begin{table}
\renewcommand{\arraystretch}{1.2}
$$\begin{array}{ll} \hline
G & H \\ \hline 
{\rm J}_{4} & 2^{11}{:}{\rm M}_{24},\; 2^{1+12}.3.{\rm M}_{22}{:}2,\; 2^{10}{:}{\rm L}_{5}(2),\; 2^{3+12}.(S_5 \times {\rm L}_{3}(2)),\; {\rm U}_{3}(11){:}2 \\
{\rm Co}_{1} & \mbox{all except $A_9 \times S_3$, $(A_7 \times {\rm L}_{2}(7)){:}2$, $(D_{10}\times (A_5 \times A_5).2).2$, $5^{1+2}{:}{\rm GL}_{2}(5)$} \\
& \mbox{$5^3{:}(4 \times A_5).2$, $7^2{:}(3 \times 2S_4)$, $5^2{:}2A_5$} \\
{\rm Fi}_{24}' & \mbox{all except $(A_9 \times A_5){:}2$, ${\rm L}_{2}(8){:}3 \times A_6$, $7{:}6 \times A_7$, ${\rm U}_{3}(3){:}2$, ${\rm L}_{2}(13){:}2$, $29{:}14$} \\
{\rm Th} & \mbox{all except $3^5{:}2S_6$, $5^{1+2}{:}4S_4$, $5^2{:}{\rm GL}_{2}(5)$, $7^2{:}(3 \times 2S_4)$, ${\rm L}_{2}(19){:}2$, ${\rm L}_{3}(3)$, ${\rm M}_{10}$, $31{:}15$, $S_5$} \\
{\rm Ru} & \mbox{all except $3.A_6.2^2$, $5^{1+2}{:}[2^5]$, ${\rm L}_{2}(13){:}2$, $A_6.2^2$, $5{:}4 \times A_5$} \\ 
\mathbb{B} & 2.{}^2E_6(2){:}2, \; 2^{1+22}.{\rm Co}_{2}, \; {\rm Fi}_{23}, \; 2^{9+16}.{\rm Sp}_{8}(2), \; {\rm Th}, \; (2^2 \times F_4(2)){:}2,\; 2^{2+10+20}.({\rm M}_{22}{:}2 \times S_3) \\
& [2^{30}].{\rm L}_{5}(2), \; S_3 \times {\rm Fi}_{22}.2, \; [2^{35}].(S_5 \times {\rm L}_{3}(2)),\; {\rm HN}.2, \; {\rm P\O}_{8}^{+}(3).S_4 \\
\mathbb{M} & 2.\mathbb{B},\; 2^{1+24}.{\rm Co}_{1},\; 3.{\rm Fi}_{24},\; 2^2.{}^2E_6(2){:}S_3,\; 2^{10+16}.\O_{10}^{+}(2),\; 2^{2+11+22}.({\rm M}_{24} \times S_3) \\
&  3^{1+12}.2{\rm Suz}.2,\; 2^{5+10+20}.(S_{3} \times {\rm L}_{5}(2)) \\ \hline
\end{array}$$
\caption{Some large maximal subgroups of sporadic groups}
\label{tab:spor}
\renewcommand{\arraystretch}{1}
\end{table}

In the statement of our next theorem, $G$ is a finite simple classical group with natural module $V$. As we will recall in Section \ref{s:class}, any maximal subgroup of $G$ belongs to one of two subgroup collections, denoted by $\C(G)$ and $\mathcal{S}(G)$. The subgroups $H \in \C(G)$ are \emph{geometric}, in the sense that they are defined in terms of the underlying geometry of $V$ (for example, $H$ may be the stabiliser of a subspace of $V$, or an appropriate tensor product decomposition of $V$). The remaining subgroups in $\mathcal{S}(G)$ are almost simple and act irreducibly on $V$ (see Definition \ref{sdef} for several other conditions that these subgroups satisfy). Following \cite{KL}, in Theorem \ref{t:main3} (and also in Theorems \ref{t:main4} and \ref{t:main5}), we refer to the \emph{type of $H$}, which provides an approximate description of the group-theoretic structure of $H$.

\begin{theorem}\label{t:main3}
Let $G$ be a finite simple classical group, and let $H \in \C(G) \cup \mathcal{S}(G)$ be a subgroup of $G$. Then $H$ is large if and only if one of the following holds:
\begin{itemize}\addtolength{\itemsep}{0.2\baselineskip}
\item[{\rm (i)}] $H \in \C(G)$ is a geometric subgroup recorded in Proposition \ref{p:psl} (for $G$ linear), \ref{p:psu} ($G$ unitary), \ref{p:psp} ($G$ symplectic) and \ref{p:pso} ($G$ orthogonal).
\item[{\rm (ii)}] $H \in \mathcal{S}(G)$ is an almost simple irreducible subgroup and $(G,H)$ is one of the cases in Table \ref{tab:sg} (see Section \ref{ss:nongeom}).
\end{itemize}
\end{theorem}

\begin{theorem}\label{t:main4}
Let $G$ be a finite simple group of exceptional Lie type, and let $H$ be a maximal subgroup of $G$. Then $H$ is large if and only if $H$ is a parabolic subgroup of $G$, or 
$(G,H)$ is one of the cases listed in Table \ref{tab:fin_ex}.
\end{theorem}

Note that in Table \ref{tab:fin_ex} we adopt the standard Lie notation for groups of Lie type, so for example we write $A_{n-1}^{-}(q)$ in place of ${\rm U}_{n}(q)$, $D_n^{-}(q)$ instead of  ${\rm P\O}_{2n}^{-}(q)$, and $E_6^{-}(q)$ for ${}^2E_6(q)$. Also note that we may assume $q>2$ if $G=G_2(q)$ since $G_2(2)' \cong {\rm U}_{3}(3)$. In this paper we view the Tits group ${}^2F_4(2)'$ as a sporadic group, so it is covered by Theorem \ref{t:main2} (in particular, every maximal subgroup of ${}^2F_4(2)'$ is large).

\begin{table}
\renewcommand{\arraystretch}{1.2}
$$\begin{array}{lll} \hline
G & \mbox{Type of $H$} & \mbox{Conditions} \\ \hline 
E_8(q) & A_1(q)E_7(q),\, D_8(q),\, A_2^{\e}(q)E_6^{\e}(q),\, E_8(q^{1/2}) & \e=\pm \\
E_7(q) & (q-\e)E_{6}^{\e}(q),\, A_1(q)D_6(q),\, A_7^{\e}(q),\, A_1(q)F_4(q),\, E_7(q^{1/2}) & \e = \pm \\
E_6^{\e}(q) &  A_1(q)A_5^{\e}(q),\, F_4(q) & \\
& (q-\e)D_5^{\e}(q) & \e=- \\
 & C_4(q) & p \neq 2 \\
 &  E_6^{\pm}(q^{1/2})  & \e=+ \\
& E_6^{\e}(q^{1/3}) & q^{1/3} \equiv \e \imod{3} \\
& (q-\e)^2.D_4(q) & (\e,q) \neq (+,2) \\
&  (q^2+\e q+1).{}^3D_4(q) & (\e,q) \neq (-,2) \\
& {\rm Fi}_{22} & (\e,q)=(-,2) \\
F_4(q) & B_4(q),\, D_4(q),\, {}^3D_4(q),\,  F_4(q^{1/2}) & \\
& A_1(q)C_3(q) & p \ne 2 \\
& C_4(q) & p = 2 \\
& {}^2F_4(q) & q=2^{2n+1} \geqs 2 \\
& {}^3D_4(2) & q=3 \\
& A_3(3) & q=2 \\
G_2(q) & A_2^{\pm}(q),\, A_1(q)^2,\, G_2(q^{1/2}) &  \\
(q>2) & {}^2G_2(q) & q=3^{2n+1} \geqs 3 \\
& {\rm J}_{1} & q=11 \\
& G_2(2) & q=5,7 \\
& A_{1}(13),\, {\rm J}_{2} & q =4 \\
& A_1(13), \, 2^{3}.{\rm SL}_{3}(2)  & q=3 \\
{}^2F_4(q) & C_2(q),\, {}^2B_2(q)^2  & \\
{}^2G_2(q) & A_1(q) & \\
{}^2B_2(q) & 13{:}4 & q=8 \\
{}^3D_4(q) & A_1(q^3)A_1(q),\, (q^2+\e q+1)A_2^{\e}(q),\, {}^3D_4(q^{1/2}),\, G_2(q)  & \e = \pm \\ 
& 7^2{:}{\rm SL}_{2}(3) & q=2 \\ \hline
\end{array}$$
\caption{Large maximal non-parabolic subgroups of finite exceptional groups}
\label{tab:fin_ex}
\renewcommand{\arraystretch}{1}
\end{table}

The proof of Theorem $i$ will be given in Section $i$ for $2 \leqs i \leqs 5$. 
For alternating groups, a bound of Mar\'{o}ti \cite{Maroti}, combined with the O'Nan-Scott theorem, essentially reduces the problem to intransitive and imprimitive subgroups, where we can compute directly with the appropriate order formulae. The proof for sporadic groups is 
an easy calculation since a complete list of maximal subgroups is available (apart from a handful of very small candidate maximals of the Monster).

The analysis for a classical group $G$ is partitioned into two cases, according to whether or not $H$ is in $\C(G)$ or $\mathcal{S}(G)$. In the former case, the precise structure of $H$ is given in \cite[Chapter 4]{KL}, so it is relatively straightforward to determine all the large subgroups, working systematically through the various subcollections comprising $\C(G)$. (Some cases require special attention; see Remark \ref{r:special}.) To determine the large subgroups in $\mathcal{S}(G)$ we first apply a theorem of Liebeck \cite{L}, which implies that the order of such a subgroup is rather small, in general. In particular, the problem is quickly reduced to classical groups of low dimension, at which point a short list of possibilities can be obtained by combining work of L\"{u}beck \cite{Lu} and Hiss, Malle \cite{HM} on the degrees of irreducible representations of quasisimple groups, and we also appeal to the recent book \cite{BHR} on the subgroup structure of the low-dimensional classical groups.

For groups of exceptional Lie type, the maximal subgroups of the low-rank groups have been determined, so the proof of Theorem \ref{t:main4} in these cases is an easy exercise. For the remaining groups, our starting point is a reduction theorem of Liebeck and Seitz \cite[Theorem 2]{LS10}, which essentially allows us to reduce to the case where $H$ is almost simple, with socle $H_0$, say. At this point there are two possibilities, which we consider separately. Write ${\rm Lie}(p)$ for the set of simple groups of Lie type in characteristic $p$, and suppose $G \in {\rm Lie}(p)$ has untwisted Lie rank $n$. If $H_0 \in {\rm Lie}(p)$ has untwisted Lie rank $r$, then the possibilities with $r>n/2$ are given by Liebeck and Seitz \cite{LS6}, but more work is needed to determine the large subgroups with $r \leqs n/2$ (an upper bound on $|H|$ given in \cite[Theorem 1.2]{LSh4} is useful here). Finally, if $H_0 \not\in {\rm Lie}(p)$ then the possibilities for $H$ are determined in \cite{LS3}, and it is straightforward to read off the large examples. 

\begin{remark}
In a short appendix we also give a complete list of the large maximal subgroups of the general linear group ${\rm GL}_{n}(q)$; see Theorem \ref{t:gln} in Appendix \ref{a:gln}.
\end{remark}

Naturally, one can extend the analysis to almost simple groups, using entirely similar methods. Let $G$ be a finite almost simple group with socle $G_0$, so we have 
$$G_0 \leqs G \leqs {\rm Aut}(G_0).$$ 
If $G_0$ is an alternating or sporadic group, then it is straightforward to determine all the large maximal subgroups of $G$; for symmetric groups, the argument in Section \ref{s:alt} goes through essentially unchanged, and one can inspect complete lists of maximal subgroups of almost simple sporadic groups (see Propositions \ref{p:almostan} and \ref{p:almostspor}). The following result for groups of Lie type may be useful in applications; a short proof is given in Section \ref{s:almost}.
% is easily obtained from our work in Sections \ref{s:class} and \ref{s:ex}:

\begin{theorem}\label{t:almost}
Let $G$ be a finite almost simple group of Lie type with socle $G_0$, and let $H$ be a maximal subgroup of $G$ such that $G=HG_0$. Then $|H \cap G_0|^3 \geqs |G_0|$ only if one of the following holds:
\begin{itemize}\addtolength{\itemsep}{0.2\baselineskip}
\item[{\rm (i)}] $(G_0,H \cap G_0)$ is one of the cases arising in Theorem \ref{t:main3} or Theorem \ref{t:main4};
\item[{\rm (ii)}] $H \cap G_0$ is a non-maximal parabolic subgroup of $G_0$;
\item[{\rm (iii)}] $(G_0,H \cap G_0)$ is one of the cases listed in Table \ref{tab:almost}.
\end{itemize}
\end{theorem}

\begin{table}
\renewcommand{\arraystretch}{1.2}
$$\begin{array}{lll} \hline
G_0 & \mbox{Type of $H \cap G_0$} & \mbox{Conditions} \\ \hline 
{\rm L}_n(q) & {\rm GL}_{m}(q) \times {\rm GL}_{n-m}(q) & 1 \leqs m < n/2 \\
\O_{10}^{-}(q) & {\rm M}_{12} & q = 2 \\
{\rm P\O}_{8}^{+}(q) & G_2(q) & \\
 & {\rm GL}_{3}^{\e}(q) \times {\rm GL}_{1}^{\e}(q) & (q,\e) \neq (3,+) \\
 & [2^9].{\rm SL}_{3}(2) & q = 3 \\
{\rm Sp}_{4}(q) & O_2^{-}(q) \wr S_2 & q = 4 \\
{\rm U}_{3}(q) & {\rm L}_{2}(7) & q = 5 \\
E_6(q) & (q-1)D_5(q) & \\
{}^2E_6(q) & {}^3D_4(q) & q = 2 \\
& B_3(3) & q = 2 \\
F_4(q) & C_2(q)^2 & p = 2 \\
& C_2(q^2) & p = 2 \\ \hline
%& S_6 \wr S_2 & q = 2 \\ \hline   
\end{array}$$
\caption{Some large subgroups of almost simple groups of Lie type}
\label{tab:almost}
\renewcommand{\arraystretch}{1}
\end{table}

%\begin{remark}
%Naturally, one could consider the analogous problem for almost simple groups. Let $G$ be a finite almost simple group with socle $G_0$, so $G_0 \leqs G \leqs {\rm Aut}(G_0)$. If $G_0$ is an alternating or sporadic group, then it is entirely straightforward to determine all the large maximal subgroups of $G$; for symmetric groups, the argument in Section \ref{s:alt} goes through essentially unchanged, and for almost simple sporadic groups one can inspect the complete lists of maximal subgroups. The situation for almost simple groups of Lie type is more complicated. For example, some maximal subgroups of $G_0$ may not extend to maximal subgroups of $G$ (this will depend on the precise structure of $G$), and the exact conditions 
%for largeness will also depend on the choice of $G$ (and some of these conditions will be rather delicate; see Lemma \ref{l:c5} in the case $G=G_0$, for example). 
%%
%%and they will depend on the precise structure of $G$.
%% , and \emph{novelty} subgroups may arise in some cases (these are maximal subgroups $H$ of $G$, with the property that $H \cap G_0$ is non-maximal in $G_0$). In addition to these technicalities, the conditions  
%These technicalities make it difficult to formulate precise results for almost simple groups of Lie type, and this explains why we have chosen to focus on simple groups in this paper.
%\end{remark}

We can use Theorems \ref{t:main1} -- \ref{t:main4} to investigate the simple groups $G$ with the property that \emph{every} maximal subgroup of $G$ is large. By Theorem \ref{t:main2}, the sporadic simple groups with this property are as follows:
$${\rm M}_{11}, {\rm M}_{12}, {\rm M}_{22}, {\rm M}_{23}, {\rm J}_{3}, {\rm HS}, {\rm McL}, {\rm Fi}_{22}, {}^2F_4(2)'.$$
If $G=A_n$ is an alternating group with $n \geqs 25$ then Corollary \ref{c:alt} (see Section \ref{s:alt}) implies that $G$ has the desired property if and only if there are no primitive permutation groups of degree $n$ (other than $A_n$ and $S_n$, of course). By the main theorem of \cite{CNT}, almost every positive integer $n$ has this property (that is, the relevant set of integers has density $1$ in $\mathbb{N}$), so there are infinitely many alternating groups such that every maximal subgroup is large (one can check that $\{5, 6, 7, 8, 9, 10, 11, 12, 15, 16, 24, 34, 39, 46\}$ are the relevant integers in the range $5 \leqs n \leqs 50$). The analogous problem for simple groups of Lie type can be studied via Theorems \ref{t:main3} and \ref{t:main4}. For example, using Theorem \ref{t:main4} we deduce that if $G$ is a simple group of exceptional Lie type then every maximal subgroup of $G$ is large if and only if $G=G_2(q)$ and $q=p^{2^i}$ ($p$ prime), where either $p \leqs 3$ and $i \geqs 0$, or $p=5$ and $i \geqs 2$.

\vs

We also extend our results to algebraic groups. Let $G$ be a connected linear algebraic group over an algebraically closed field of characteristic $p \geqs 0$. If $G$ admits a triple factorisation $G = ABA$, where $A,B$ are closed subgroups of $G$, then it is straightforward to show that 
$$\max\{3\dim A , 3\dim B\} \geqs \dim G$$
(see Proposition \ref{p:ta}) which is a natural algebraic group analogue of the condition in \eqref{e:max}. Therefore, in this context we will say that a proper closed subgroup $H$ of $G$ is \emph{large} if $3\dim H \geqs \dim G$. We classify all the large closed maximal subgroups of simple algebraic groups. 

\begin{theorem}\label{t:main5}
Let $G$ be a simple algebraic group over an algebraically closed field of characteristic $p \geqs 0$ and let $H$ be a maximal closed subgroup of $G$. Then $H$ is large if and only if one of the following holds:
\begin{itemize}\addtolength{\itemsep}{0.2\baselineskip}
\item[{\rm (i)}] $G=Cl(V)$ is a classical group and either $H$ acts reducibly on $V$, or $(G,H)$ is one of the cases listed in Table \ref{tab:alg_class}.
\item[{\rm (ii)}] $G$ is an exceptional group and either $H$ is a parabolic subgroup, or $(G,H)$ is one of the cases listed in Table \ref{tab:alg_ex}.
\end{itemize}
\end{theorem}

\begin{table}
$$\begin{array}{lll} \hline
G & \mbox{Type of $H$} & \mbox{Conditions} \\ \hline
{\rm SL}_{n} & {\rm GL}_{n/2} \wr S_2 & \mbox{$n$ even} \\
(n \geqs 2) & 
  {\rm Sp}_{n} & \mbox{$n \geqs 4$ even} \\
 & {\rm SO}_{n} & n \geqs 2, \, p \neq 2 \\
{\rm Sp}_{n} & {\rm Sp}_{n/t} \wr S_t & \mbox{$t=2,3$ or $(n,t) = (8,4)$} \\
(n \geqs 4) & {\rm GL}_{n/2} & p \ne 2 \\
& O_{n} & p=2 \\
& G_2 & (n,p)=(6,2) \\
{\rm SO}_{n} & O_{n/2} \wr S_2 & \mbox{$n$ even} \\
(n \geqs 7) & {\rm GL}_{n/2} & \\
& {\rm Sp}_{a} \otimes {\rm Sp}_{n/a} & (n,a)=(8,2), (12,2) \\
& G_2 & n=7, \, p \neq 2 \\
& B_3 & n=8 \\
& C_3 & (n,p)=(8,2) \\ \hline
\end{array}$$
\caption{Large irreducible maximal subgroups of classical algebraic groups}
\label{tab:alg_class}
\end{table}

\begin{table}
$$\begin{array}{ll} \hline
G & H  \\ \hline
E_8 & D_8,\, A_1E_7,\, A_2E_6.2   \\
E_7 & T_1E_6.2,\, A_1D_6,\, A_7.2,\, A_1F_4  \\
E_6 & A_1A_5,\,  T_2D_4.S_3,\, F_4,\, C_4\, (p\neq 2) \\
F_4 & B_4,\, D_4.S_3,\, C_4\, (p=2),\,  A_1C_3\, (p \neq 2) \\ 
G_2 & A_2.2,\, A_1A_1 \\ \hline
\end{array}$$
\caption{Large non-parabolic maximal subgroups of exceptional algebraic groups}
\label{tab:alg_ex}
\end{table}

For classical algebraic groups, our main tool is a reduction theorem of Liebeck and Seitz \cite{LS4}, which provides an algebraic group analogue of Aschbacher's theorem for finite classical groups. 
For exceptional algebraic groups, the positive-dimensional maximal closed subgroups are determined in \cite{LS2}, and it is straightforward to read off the large examples.

\vs

Finally, some brief comments on the notation used in this paper. Our group-theoretic notation is standard, and it is consistent with the notation in \cite{KL}. In particular, we write ${\rm L}_{n}(q) = {\rm L}_{n}^{+}(q) = {\rm PSL}_{n}(q)$ and ${\rm U}_{n}(q) = {\rm L}_{n}^{-}(q) = {\rm PSU}_{n}(q)$, and the simple orthogonal groups are denoted by ${\rm P\O}_{n}^{\pm}(q)$ ($n$ even), and $\O_n(q)$ ($n$ odd); we also use the notation $\O_n^{\circ}(q)$ in the latter case. As previously noted, when working with exceptional groups (both finite and algebraic), it is convenient to adopt the standard Lie notation.
In addition, if $a,b$ are positive integers then we write $(a,b)$ for the greatest common divisor of $a$ and $b$, and $a_b$ denotes the largest $b$-power dividing $a$. 
We refer the reader to \cite[pp.170--171]{KL} for a convenient list of formulae for the orders of all finite simple groups.

\section{Alternating groups}\label{s:alt}

Let $G=A_n$ be an alternating group of degree $n \geqslant 5$, and let $H$ be a maximal subgroup of $G$. Then one of the following holds:
\begin{itemize}\addtolength{\itemsep}{0.2\baselineskip}
\item[{\rm (a)}] $H=(S_k \times S_{n-k}) \cap G$ is intransitive on $\{1, \ldots, n\}$, where $1 \leqslant k <n/2$;
\item[{\rm (b)}] $H=(S_k \wr S_{n/k}) \cap G$ is transitive but imprimitive on $\{1, \ldots, n\}$, where $k$ divides $n$ and $2 \leqslant k \leqslant n/2$;
\item[{\rm (c)}] $H$ is primitive on $\{1, \ldots, n\}$.
\end{itemize}

It is easy to check that all subgroups of type (a) or (b) are large. For example, if 
$H$ is of type (a) then
$$|H| = \frac{1}{2}k!(n-k)!>\frac{1}{2}(\lfloor n/2 \rfloor!)^2>\frac{1}{2}((n-1)/2e)^{n-1}$$
and this bound implies that $|H|^3>n^n>|G|$ for all $n \geqslant 17$. The cases with $5\leqs n<17$ can be checked directly. A similar argument applies in (b).

The possibilities in case (c) are described by the O'Nan-Scott theorem (see \cite[p.366]{LPS}, for example). By combining this result with a theorem of Mar\'{o}ti \cite[Theorem 1.1]{Maroti}, we get the following:

\begin{prop}\label{p:maroti}
Let $H$ be a maximal primitive subgroup of $G=A_n$. Then one of the following holds:
\begin{itemize}\addtolength{\itemsep}{0.2\baselineskip}
\item[{\rm (i)}] $H = (S_a \wr S_b) \cap G$ and $n=a^b$, with $a \geqs 5$ and $b \geqs 2$;
\item[{\rm (ii)}] $H = S_a \cap G$ and $n = \binom{a}{b}$, with $a \geqs 5$ and $2 \leqs b \leqs a-2$;
\item[{\rm (iii)}] $(n,H) = (11,{\rm M}_{11}), (12,{\rm M}_{12}), (23, {\rm M}_{23})$ or $(24, {\rm M}_{24})$;
\item[{\rm (iv)}] $|H| < n^{1+\lfloor \log_{2}n \rfloor}$.
\end{itemize}
\end{prop}

\begin{cor}\label{c:alt}
Let $H$ be a large primitive maximal subgroup of $G=A_n$, where $n \geqs 24$. Then $(n,H) = (24, {\rm M}_{24})$.
\end{cor}

\begin{proof}
This is entirely straightforward. For example, in case (ii) we have
$$|H|^3<a^{3a},\;\; |G| \geqs \frac{1}{2}\binom{a}{2}!>\frac{1}{2}\left(\frac{a(a-1)}{2e}\right)^{\frac{1}{2}a(a-1)}$$
so $|H|^3<|G|$ if
$$f(a):=\frac{1}{2}a(a-1)\left(\log(a(a-1))-\log(2e)\right)-3a\log a -\log 2 > 0.$$
For $a \geqs 5$, it is easy to check that $f$ is increasing, and we have $f(7)>0$. The cases $a=5,6$ can be checked directly.
\end{proof}

Finally, the remaining cases with $n \leqslant 23$ can be checked directly, using the complete list of maximal subgroups of $A_n$ available in {\sc Magma} \cite{magma}, for example. In this way we obtain the list of large subgroups presented in Theorem \ref{t:main1}.

\section{Sporadic groups}\label{s:spot}

The proof of Theorem \ref{t:main2} is an entirely straightforward exercise. Indeed, all the maximal subgroups of sporadic groups have been determined up to conjugacy, with the exception of the Monster (see \cite{WebAt} for a convenient list of these subgroups). At present, $44$ conjugacy classes of maximal subgroups of the Monster have been determined, and it is known that any additional maximal subgroup is almost simple, with socle ${\rm L}_{2}(13)$, ${\rm U}_{3}(4)$, ${\rm U}_{3}(8)$ or ${\rm Sz}(8)$ (see \cite{N2}), which is clearly non-large.

\section{Finite classical groups}\label{s:class}

Let $G$ be a finite simple classical group over $\F$, with natural (projective) module $V$ of dimension $n$. Write $q=p^a$ for a prime $p$. The possibilities for $G$ are as follows:
$${\rm L}_{n}(q),\, n \geqs 2;\;\; {\rm U}_{n}(q),\, n \geqs 3;\;\; {\rm PSp}_{n}(q)',\,n \geqs 4;\;\; {\rm P\O}_{n}^{\e}(q),\, n \geqs 7.$$
Due to the existence of isomorphisms between certain low-dimensional classical groups (see \cite[Proposition 2.9.1]{KL}, for example), we may assume that $n$ satisfies the stated lower bounds (in addition, we will also assume that $q$ is odd if $G$ is an odd-dimensional orthogonal group). We begin by recording some useful preliminary results. 

\subsection{Preliminaries}\label{s:cprel}

\begin{lem}\label{l:ord2}
Let $q$ be a prime power.
\begin{itemize}\addtolength{\itemsep}{0.2\baselineskip}
\item[{\rm (i)}] If $a \geqs 2$ then
$$1-q^{-1}-q^{-2}<\prod_{i=1}^a(1-q^{-i}) \leqs (1-q^{-1})(1-q^{-2}).$$
\item[{\rm (ii)}] If $a \geqs 3$ then
$$1<(1+q^{-1})(1-q^{-2})<\prod_{i=1}^a(1-(-q)^{-i})\leqs (1+q^{-1})(1-q^{-2})(1+q^{-3}).$$
\end{itemize}
\end{lem}

\begin{proof}
Part (i) follows from \cite[Lemma 3.5]{NP}. For (ii), observe that $(1+q^{-2m+1})(1-q^{-2m})>1$ and $(1-q^{-2m})(1+q^{-2m-1})<1$ for all $m \geqs 1$.
\end{proof}

\begin{lem}\label{l:ord4}
\mbox{ }
\begin{itemize}\addtolength{\itemsep}{0.2\baselineskip}
\item[{\rm (i)}] If $a,q \geqs 2$ then
\begin{align*}
(1-q^{-1}-q^{-2})q^{a^2} &< |{\rm GL}_{a}(q)| \leqs (1-q^{-1})(1-q^{-2})q^{a^2} \\
(1+q^{-1})(1-q^{-2})q^{a^2} &\leqs  |{\rm GU}_{a}(q)| \leqs (1+q^{-1})(1-q^{-2})(1+q^{-3})q^{a^2}
\end{align*}
\item[{\rm (ii)}] If $a \geqs 4$ and $q \geqs 2$ then
\begin{align*}
(1-q^{-2}-q^{-4})q^{\frac{1}{2}a(a+1)} &< |{\rm Sp}_{a}(q)| \leqs (1-q^{-2})(1-q^{-4})q^{\frac{1}{2}a(a+1)}
\end{align*}
\item[{\rm (iii)}] If $a \geqs 5$ and $q \geqs 2$ then
\begin{align*}
(1-q^{-2}-q^{-4})q^{\frac{1}{2}a(a-1)} &< \a^{-1}|{\rm SO}_{a}^{\circ}(q)| \leqs (1-q^{-2})(1-q^{-4})q^{\frac{1}{2}a(a-1)} \\
(1-q^{-2}-q^{-4})(1-q^{-a/2})q^{\frac{1}{2}a(a-1)} &< \a^{-1}|{\rm SO}_{a}^{+}(q)| \leqs (1-q^{-2})(1-q^{-4})q^{\frac{1}{2}a(a-1)}  \\
(1-q^{-2}-q^{-4})q^{\frac{1}{2}a(a-1)} &< \a^{-1}|{\rm SO}_{a}^{-}(q)| \leqs (1-q^{-2})(1-q^{-4})(1+q^{-a/2})q^{\frac{1}{2}a(a-1)}  
\end{align*}
where $\a=(2,q)$.
\end{itemize}
\end{lem}

\begin{proof}
We have
\renewcommand{\arraystretch}{1.2}
$$\begin{array}{ll}
\displaystyle |{\rm GL}_{a}(q)| = q^{a^2}\prod_{i=1}^{a}(1-q^{-i}) & 
\displaystyle |{\rm GU}_{a}(q)|  = q^{a^2}\prod_{i=1}^{a}(1-(-q)^{-i}) \\
\displaystyle |{\rm Sp}_{a}(q)|  = q^{\frac{1}{2}a(a+1)}\prod_{i=1}^{a/2}(1-q^{-2i}) & 
\displaystyle |{\rm SO}_{a}^{\circ}(q)|  = \a q^{\frac{1}{2}a(a-1)}\prod_{i=1}^{(a-1)/2}(1-q^{-2i}) \\ 
\multicolumn{2}{l}{\displaystyle |{\rm SO}_{a}^{\pm}(q)|  = \a q^{\frac{1}{2}a(a-1)}(1\mp q^{-a/2})\prod_{i=1}^{a/2-1}(1-q^{-2i})} \\
\end{array}$$
\renewcommand{\arraystretch}{1}
so all of the bounds follow immediately from Lemma \ref{l:ord2}.
\end{proof}

\begin{cor}\label{l:ord}
\mbox{ }
\begin{itemize}\addtolength{\itemsep}{0.2\baselineskip}
\item[{\rm (i)}] If $n \geqs 2$ then 
$$q^{n^2-2}<|{\rm L}_{n}(q)| \leqs (1-q^{-2})q^{n^2-1}$$
\item[{\rm (ii)}] If $n \geqs 3$ then 
$$(1-q^{-1})q^{n^2-2} < |{\rm U}_{n}(q)| \leqs (1-q^{-2})(1+q^{-3})q^{n^2-1}<q^{n^2-1}$$
\item[{\rm (iii)}] If $n \geqs 4$ then 
$$\frac{1}{2\a}q^{\frac{1}{2}n(n+1)}<|{\rm PSp}_{n}(q)| < q^{\frac{1}{2}n(n+1)}$$
with $\a=(2,q-1)$.
\item[{\rm (iv)}] If $n \geqs 7$ then
$$\frac{1}{4\b}q^{\frac{1}{2}n(n-1)}<|{\rm P\O}_{n}^{\e}(q)| < q^{\frac{1}{2}n(n-1)}$$
with $\b=(2,n)$.
\end{itemize}
In particular, $|G|>\frac{1}{8}q^{\frac{1}{2}n(n-1)}$ for every finite simple classical group $G$ of dimension $n$ over $\F$.
\end{cor}

\begin{proof}
This quickly follows from the bounds in Lemma \ref{l:ord4}. For example,
$$|{\rm L}_{n}(q)| = \frac{|{\rm GL}_{n}(q)|}{d(q-1)} > \frac{(1-q^{-1}-q^{-2})}{d(q-1)}q^{n^2} \geqslant \frac{q^2-q-1}{(q-1)^2}q^{n^2-2} \geqslant q^{n^2-2},$$
where $d=(n,q-1)$, and
$$|{\rm L}_{n}(q)| = \frac{|{\rm GL}_{n}(q)|}{d(q-1)} \leqs d^{-1}(1-q^{-2})q^{n^2-1} \leqs (1-q^{-2})q^{n^2-1}.$$
The other cases are entirely similar.
\end{proof}

We will also need some elementary bounds on factorials.

\begin{lem}\label{l:ord3}
The following bounds hold:
\begin{itemize}\addtolength{\itemsep}{0.2\baselineskip}
\item[{\rm (i)}] If $t \geqs 5$ then $t!<5^{(t^2-3t+1)/3}$.
\item[{\rm (ii)}] If $t \geqs 4$ then $t!<2^{4t(t-3)/3}$.
\end{itemize}
\end{lem}

\begin{proof}
First consider (i). If $t=5$ or $6$ then the bound can be checked directly, so let us assume $t \geqs 7$. Since $t!<t^t$, the desired bound holds if 
$f(t) = \frac{1}{3}(t^2-3t+1)\log 5 - t\log t>0$. For $t \geqs 7$ we have $f'(t) = (\frac{2}{3}t-1)\log 5-\log t - 1>0$, so $f(t) \geqs f(7)>0$ and the result follows. A similar argument applies in part (ii). 
\end{proof}

\subsection{Subgroup structure}\label{ss:substr}

The main theorem on the subgroup structure of a finite classical group $G$ is due to Aschbacher. In \cite{asch}, eight natural, or \emph{geometric}, subgroup collections of $G$ are defined, denoted $\C_i$, $1 \leqs i \leqs 8$. These collections include the stabilisers of subspaces of $V$, and the stabilisers of appropriate direct sum and tensor product decompositions of $V$. We follow \cite{KL} in labelling the various $\C_i$ collections, and we refer the reader to \cite[Chapter 4]{KL} for a detailed description of the relevant subgroups. We write
$$\mathcal{C}(G) = \bigcup_{i=1}^{8}\C_{i}.$$

Aschbacher's main theorem in \cite{asch} states that if $H$ is a maximal subgroup of $G$ then either $H$ belongs to $\C(G)$, or $H$ is almost simple, with socle $H_0$ say, and there exists an absolutely irreducible representation $\rho: \what{H}_{0} \to {\rm SL}(V)$, where $\what{H}_{0}$ is the full covering group of $H_0$. We write $\mathcal{S}(G)$ to denote this family of almost simple irreducible subgroups of $G$, and we refer the reader to Definition \ref{sdef} for some additional conditions satisfied by these subgroups (these conditions are imposed to ensure that a subgroup in $\mathcal{S}(G)$ is not contained in one of the geometric $\mathcal{C}_i$ collections). A rough description of the $\C_i$ families is given in Table \ref{t:max}.

\begin{table}[h]
$$\begin{array}{ll} \hline
\C_1 & \mbox{Stabilisers of subspaces of $V$} \\
\C_2 & \mbox{Stabilisers of decompositions $V=\bigoplus_{i=1}^{t}V_i$, where $\dim V_i  = a$} \\
\C_3 & \mbox{Stabilisers of prime index extension fields of $\F$} \\
\C_4 & \mbox{Stabilisers of decompositions $V=V_1 \otimes V_2$} \\
\C_5 & \mbox{Stabilisers of prime index subfields of $\F$} \\
\C_6 & \mbox{Normalisers of symplectic-type $r$-groups in absolutely irreducible representations} \\
\C_7 & \mbox{Stabilisers of decompositions $V=\bigotimes_{i=1}^{t}V_i$, where $\dim V_i  = a$} \\
\C_8 & \mbox{Stabilisers of non-degenerate forms on $V$} \\ \hline
\end{array}$$
\caption{The geometric subgroup collections}
\label{t:max}
\end{table}

In view of Aschbacher's theorem, it is natural to partition the proof of Theorem \ref{t:main3}  into two cases. In Section \ref{ss:geom} we deal with the geometric subgroups  $H \in \C(G)$, in which case the order of $H$ can be readily computed from the information in \cite[Chapter 4]{KL}. Next, in Section \ref{ss:nongeom}, we turn our attention to the almost simple irreducible subgroups in $\mathcal{S}(G)$. A key tool here is a theorem of Liebeck \cite{L}, which states that if $H \in \mathcal{S}(G)$ has socle $H_0$ then either $|H|<q^{3n}$, or $H_0$ is an alternating group and $V$ is the fully deleted permutation module for $H_0$ (see Theorem \ref{lieb}).

\begin{remk}\label{r:max}
In our analysis of the subgroups in $\C(G)$, we will adopt the precise definition of the $\C_i$ collections given in \cite{KL}. In particular, the subgroups in $\C(G)$ are listed in \cite[Tables 3.5.A -- 3.5.F]{KL}, and we adopt the conditions recorded in column IV of these tables. Moreover, since $G$ is simple, in our definition of the $\C_i$ collections we will also exclude any \emph{novelties} that are recorded in column VI of these tables. So for example, if $G={\rm L}_{n}(q)$ and $H \in \C_1$ then we will assume that $H=P_i$ for some $i$ (that is, we exclude the non-maximal subgroups of type $P_{i,n-i}$ and ${\rm GL}_{i}(q) \oplus {\rm GL}_{n-i}(q)$ listed in \cite[Table 3.5.A]{KL}). 
\end{remk}

\subsection{Geometric subgroups}\label{ss:geom}

We start by determining the large subgroups in $\C(G)$. 

\begin{remk}\label{r:special}
The following three cases will require special attention:
\begin{itemize}\addtolength{\itemsep}{0.2\baselineskip}
\item[{\rm (i)}] $H$ is a $\C_2$-subgroup stabilising a direct sum decomposition $V = V_1 \oplus V_2 \oplus V_3$;
\item[{\rm (ii)}] $H$ is a $\C_3$-subgroup corresponding to a degree-three field extension of $\F$;
\item[{\rm (iii)}] $H$ is a $\C_5$-subgroup corresponding to an index-three subfield of $\F$.
\end{itemize}
Indeed, in each of these cases, $|H|^3 \approx |G|$ and the corresponding conditions for largeness are rather delicate.
\end{remk}

We consider each of the classical groups in turn, beginning with the linear groups. In the statement of our main results, if $m$ is an integer and $p$ is a prime, then $m_p$ denotes the  largest power of $p$ dividing $m$.

\subsubsection{Linear groups}\label{sss:linear}

\begin{prop}\label{p:psl}
Let $G={\rm L}_{n}(q)$ and let $H \in \mathcal{C}(G)$ be a geometric subgroup of $G$. Then $H$ is large if and only if one of the following holds:
\begin{itemize}\addtolength{\itemsep}{0.2\baselineskip}
\item[{\rm (i)}] $H \in \C_1 \cup \C_8$;
\item[{\rm (ii)}] $H$ is a $\C_2$-subgroup of type ${\rm GL}_{n/t}(q)\wr S_t$, where $t=2$, or $t=3$ and either $q \in \{5,8,9\}$ and $(n,q-1)=1$, or $(n,q)=(3,11)$;
\item[{\rm (iii)}] $H$ is a $\C_3$-subgroup of type ${\rm GL}_{n/k}(q^k)$, where $k=2$, or $k=3$ and either $q \in \{2,3\}$, or $q=5$ and $n$ is odd.
\item[{\rm (iv)}] $H$ is a $\C_5$-subgroup of type ${\rm GL}_{n}(q_0)$ with $q=q_0^k$, and either $k=2$, or $k=3$ and $f>1$, where
\renewcommand{\arraystretch}{1.7}
$$f = \left\{\begin{array}{ll} 
\frac{27(n,q_0^3-1)}{(n,q_0-1)^3} & \mbox{if $(q_0^2+q_0+1)_3>1$ and $(q_0-1)_3 \geqslant n_3>1$} \\
\frac{(n,q_0^3-1)}{(n,q_0-1)^3} & \mbox{otherwise.}
\end{array}\right.$$
\renewcommand{\arraystretch}{1}
\item[{\rm (v)}] $H$ is a $\C_6$-subgroup and $(G,H)$ is one of the following:
$$({\rm L}_{4}(5), 2^4.A_6),\; ({\rm L}_{3}(4),3^2.Q_8),\; ({\rm L}_{2}(23),S_4), \; ({\rm L}_{2}(17),S_4),$$
$$({\rm L}_{2}(13),A_4), \; ({\rm L}_{2}(11),A_4), \; ({\rm L}_{2}(7),S_4), \; ({\rm L}_{2}(5),A_4).$$
\end{itemize}
\end{prop}

\begin{remk}\label{r:psl}
In part (v) of Proposition \ref{p:psl}, the subgroup $A_4<{\rm L}_{2}(11)$ is non-maximal.
\end{remk}

Set $d=(n,q-1)$. 
By Corollary \ref{l:ord}(i) we have
$$q^{n^2-2}<|G| \leqs (1-q^{-2})q^{n^2-1}<q^{n^2-1}.$$
We consider each $\C_i$ collection in turn. 

\begin{lem}\label{l:c1}
The conclusion to Proposition \ref{p:psl} holds if $H \in \C_1$.
\end{lem}

\begin{proof}
Here $H=P_i$ is a maximal parabolic subgroup of $G$, where $1 \leqslant i < n$ (see Remark \ref{r:max}). By \cite[Proposition 4.1.17]{KL} we have $|H|  = d^{-1}q^{i(n-i)}(q-1)|{\rm SL}_{i}(q)||{\rm SL}_{n-i}(q)|$, and by applying Lemma \ref{l:ord4}(i) we get
$$|H| > \frac{1}{4}q^{n^2-ni+i^2-2} \geqs \frac{1}{4}q^{\frac{3}{4}n^2-2}=: \a.$$
It is easy to check that 
\begin{equation}\label{e:al}
\a^3>q^{n^2-1}>|G|
\end{equation}
for all $n \geqs 3$. Finally, if $n=2$ then $|H|=q(q-1)/d$, $|G| = q(q^2-1)/d$ and thus $|H|^3>|G|$ for all $q \geqs 4$. 
\end{proof}

\begin{lem}\label{l:c2}
The conclusion to Proposition \ref{p:psl} holds if $H \in \C_2$.
\end{lem}

\begin{proof}
Let $H \in \C_2$ be a subgroup of type ${\rm GL}_{n/t}(q) \wr S_t$, where $t \geqs 2$. Note that $q \geqs 5$ if $n=t$, and $q \geqs 3$ if $n=2t$ (see \cite[Table 3.5.A]{KL} and \cite[Proposition 2.3.6]{BHR}). According to \cite[Proposition 4.2.9]{KL} we have
$$|H| = d^{-1}(q-1)^{t-1}|{\rm SL}_{n/t}(q)|^{t}t!$$
If $t=2$ then Lemma \ref{l:ord4}(i) implies that $|H|>\frac{1}{2}q^{n^2/2-2}=:\a$ and we deduce that \eqref{e:al} holds if $n \geqs 4$ (note that $H$ is non-maximal if $(n,q) = (4,2)$).
If $n=2$ then $|H|=(2,q)(q-1)$ and once again we deduce that $H$ is large. 

Next suppose $t \geqs 4$. Since $|H|<q^{n^2/t-1}t!$ and $|G|>q^{n^2-2}$, it follows that $|H|^3<|G|$ if
\begin{equation}\label{e:22}
(t!)^3<q^{n^2\left(1-\frac{3}{t}\right)+1}.
\end{equation}
If $n=t$ then $q \geqs 5$, so Lemma \ref{l:ord3}(i) implies that \eqref{e:22} holds if $t \geqs 5$; the case $t=4$ can be checked directly. If $n \geqs 2t$ then \eqref{e:22} follows from the bound in Lemma \ref{l:ord3}(ii). We conclude that $H$ is non-large if $t \geqs 4$. 

To complete the proof of the lemma, we may assume that $t=3$. Here $|H|^3<|G|$ if and only if 
\begin{equation}\label{e:psl2}
\frac{|{\rm GL}_{n/3}(q)|^9}{|{\rm GL}_{n}(q)|} < \frac{(q-1)^2d^2}{6^3}.
\end{equation}
If $n=3$ then $q \geqs 5$ and we calculate that $H$ is large if and only if $q \in \{5,8,9,11\}$. Now assume $n \geqs 6$. In view of Lemma \ref{l:ord4}(i) we have 
$$f(q):=\frac{(1-q^{-1}-q^{-2})^9}{(1-q^{-1})(1-q^{-2})} < \frac{|{\rm GL}_{n/3}(q)|^9}{|{\rm GL}_{n}(q)|} \leqs \frac{(1-q^{-1})^9(1-q^{-2})^9}{1-q^{-1}-q^{-2}}=:g(q)$$
and it is easy to check that $g(q)<1$. In particular, $H$ is non-large if $q \geqs 13$ since 
$(q-1)^2d^2/6^3 \geqs 32/27$ (minimal if $q=17$ and $d=1$).
Similarly, $g(11)<(11-1)^2/6^3 = 25/54$, so $H$ is also non-large if $q=11$. In exactly the same way, we deduce that $|H|^3<|G|$ if $q \in \{2,3,4,7\}$ (note that $d \geqs 3$ if $q=4$ or $7$). Finally, suppose $q \in \{5,8,9\}$. If $d \geqs 2$ then $g(q)<(q-1)^2d^2/6^3$ and thus $H$ is non-large. However, if $d=1$ then $H$ is large since $f(q) > (q-1)^2/6^3$.
\end{proof}

\begin{lem}\label{l:c3}
The conclusion to Proposition \ref{p:psl} holds if $H \in \C_3$.
\end{lem}

\begin{proof}
According to \cite[Proposition 4.3.6]{KL}, 
$$|H| = d^{-1}(q-1)^{-1}|{\rm GL}_{n/k}(q^k)|k,$$
where $k$ is a prime divisor of $n$, so $|H|^3<|G|$ if and only if 
$$\frac{|{\rm GL}_{n/k}(q^k)|^3}{|{\rm GL}_{n}(q)|} <\frac{(q-1)^2d^2}{k^3}.$$

First assume $n=k$. If $n=2$ then $|H|=(2,q)(q+1)$ and we quickly deduce that $H$ is large. Similarly, if $n=3$ then $H$ is large if and only if $q \in \{2,3,5\}$. If $n=k \geqs 5$ then $|H|<2nq^{n-1}$ and the bound $|H|^3<q^{n^2-2}<|G|$ follows.

For the remainder, we may assume that $n \geqs 2k$. If $k=2$ then $H$ is large since $|H|>q^{n^2/2-2}$. On the other hand, if $k \geqs 5$ then $|H|<5q^{n^2/5}$ and thus $|H|^3<|G|$. Finally, suppose $k=3$. By applying the bounds in Lemma \ref{l:ord4}(i) we get
$$f(q):=\frac{(1-q^{-3}-q^{-6})^3}{(1-q^{-1})(1-q^{-2})}<\frac{|{\rm GL}_{n/3}(q^3)|^3}{|{\rm GL}_{n}(q)|}<\frac{(1-q^{-3})^3(1-q^{-6})^3}{1-q^{-1}-q^{-2}}=:g(q).$$
If $q \geqs 7$ then $g(q) \leqs g(7)<4/3$ and thus $H$ is non-large. Similarly, if $q=4$ then $d=3$ and $g(4)<3$, so $|H|^3<|G|$. Next suppose $q=5$. If $n$ is even then $|H|^3<|G|$ since $d \geqs 2$ and $g(5)<64/27$. However, if $n$ is odd then $d=1$ and $f(5)>16/27$, so $H$ is large in this case. In exactly the same way, we deduce that $H$ is large if $q=2$ or $3$.
\end{proof}

\begin{lem}\label{l:c4}
The conclusion to Proposition \ref{p:psl} holds if $H \in \C_4$.
\end{lem}

\begin{proof}
Here $H$ is a tensor product subgroup of type ${\rm GL}_{a}(q) \otimes {\rm GL}_{n/a}(q)$ with $2 \leqslant a <n/2$, and \cite[Proposition 4.4.10]{KL} states that
$$|H| = d^{-1}|{\rm SL}_{a}(q)||{\rm SL}_{n/a}(q)|(q-1,a,n/a).$$
Therefore, 
$|H|<q^{(n/a)^2+a^2-2} \leqs q^{n^2/4+2}$
and thus $|H|^3< q^{3n^2/4+6} < q^{n^2-2}<|G|$.
\end{proof}

\begin{lem}\label{l:c5}
The conclusion to Proposition \ref{p:psl} holds if $H \in \C_5$.
\end{lem}

\begin{proof}
Let $H$ be a subfield subgroup of $G$ of type ${\rm GL}_{n}(q_0)$, where $q=q_0^k$ for a prime $k$. According to \cite[Proposition 4.5.3]{KL} we have
$$|H| = (q_0-1)^{-1}\left(q_0-1, (q_0^k-1)d^{-1}\right)|{\rm SL}_{n}(q_0)| \geqs |{\rm L}_{n}(q_0)|.$$
If $k \geqs 5$ then $H$ is non-large since
$$|H|^3 \leqs |{\rm SL}_{n}(q_0)|^3 <q_0^{3(n^2-1)} \leqs q^{\frac{3}{5}(n^2-1)} <q^{n^2-2}<|G|.$$
On the other hand, if $k=2$ and $n \geqs 3$ then $|H|>\frac{1}{2}q^{n^2/2-1}=:\a$ and \eqref{e:al} holds. 
Similarly, if $n=k=2$ then $|H|^3=q^{3/2}(q-1)^3>q^3>|G|$. Therefore, $H$ is large if $k=2$.

Finally, let us assume $k=3$. By \cite[Proposition 4.5.3]{KL}, we have $H={\rm L}_{n}(q_0).e \leqs {\rm PGL}_{n}(q_0)$, where 
$$e = \left(q_0-1,\frac{q_0^3-1}{(n,q_0^3-1)}\right)(n,q_0-1)(q_0-1)^{-1}$$
and thus $|H|^3<|G|$ if and only if
$$\zeta := \frac{|{\rm GL}_{n}(q_0^3)|}{|{\rm GL}_{n}(q_0)|^3} \cdot \frac{(q_0-1)^3}{q_0^3-1}>e^3\frac{(n,q_0^3-1)}{(n,q_0-1)^3} =: f.$$
For an integer $m$, let $m_3$ denote the largest power of $3$ dividing $m$. Using \cite[Lemma B.6]{BG_book} we calculate that 
\renewcommand{\arraystretch}{1.7}
$$e = \left\{\begin{array}{ll} 
3 & \mbox{if $(q_0^2+q_0+1)_3>1$ and $(q_0-1)_3 \geqslant n_3>1$} \\
1 & \mbox{otherwise.}
\end{array}\right.$$
\renewcommand{\arraystretch}{1}

We claim that $H$ is large if and only if $f>1$. Applying Lemma \ref{l:ord4}(i) we deduce that 
$$\zeta >\frac{(1-q_0^{-3}-q_0^{-6})}{(1-q_0^{-1})^3(1-q_0^{-2})^3}\frac{(q_0-1)^3}{q_0^3-1} >1$$
for all $q_0 \geqs 2$, whence $|H|^3<|G|$ if $f \leqs 1$. It remains to show that $H$ is large if $f>1$. Again, using the usual bounds we deduce that
$$\zeta < \frac{(1-q_0^{-3})(1-q_0^{-6})}{(1-q_0^{-1}-q_0^{-2})^3}\frac{(q_0-1)^3}{q_0^3-1} =:g(q_0),$$
so $H$ is large if $f \geqs g(q_0)$. Note that $g$ is a decreasing function. 

First assume $q_0=2$, so $f=(n,7)$ and the condition $f>1$ implies that $f=7$, so $n$ is divisible by $7$. Then $|{\rm GL}_{n}(2)|>(1-2^{-1}-2^{-2}+2^{-5})2^{n^2}$ (see the proof of \cite[Lemma 3.5]{NP}), so 
$$\zeta = \frac{|{\rm GL}_{n}(8)|}{7|{\rm GL}_{n}(2)|^3} < \frac{(1-8^{-1})(1-8^{-2})}{7(1-2^{-1}-2^{-2}+2^{-5})^3}<7=f$$
and thus $H$ is large. Now suppose $q_0 \geqslant 3$. There are two cases to consider. First assume $e=3$. Then $(n,q_0-1) = 3\ell$ and $(n,q_0^3-1) = 3\ell m$ for some $\ell, m \geqs 1$, so $f = 3m/\ell^2$ and $q_0 \geqs 3\ell+1$. Since $f>1$ we deduce that 
\begin{equation}\label{e:pp}
f \geqs \frac{\ell^2+1}{\ell^2} \geqs g(3\ell+1) \geqs g(q_0)>\zeta
\end{equation}
for all $\ell \geqs 1$, and thus $H$ is large. Finally, let us assume $e=1$ and set $(n,q_0-1) = \ell$ and $(n,q_0^3-1) = \ell m$, so $f = m/\ell^2$ and $q_0 = a\ell+1$ for some $a \geqs 1$. Note that $m$ divides $q_0^2+q_0+1$. If $a \geqs 2$ then \eqref{e:pp} holds (with $g(3\ell+1)$ replaced by $g(2\ell+1)$). Finally, suppose $a=1$, so $\ell \geqs 2$ since we are assuming that $q_0 \geqslant 3$. If $m \geqs \ell^2+3$ then
$$f \geqs \frac{\ell^2+3}{\ell^2} \geqs g(\ell+1) \geqs g(q_0)>\zeta.$$
On the other hand, if $a=1$ and $m \in \{\ell^2+1, \ell^2+2\}$ then it is easy to check that $m$ does not divide $q_0^2+q_0+1$, so this situation does not arise. We conclude that $H$ is large if $f>1$.
\end{proof}

\begin{lem}\label{l:c6}
The conclusion to Proposition \ref{p:psl} holds if $H \in \C_6$.
\end{lem}

\begin{proof}
There are three cases to consider. First assume 
$H$ is of type $r^{2m}.{\rm Sp}_{2m}(r)$, where $n=r^m$, $m \geqs 1$ and $r$ is an odd prime divisor of $q-1$. By applying \cite[Proposition 4.6.5]{KL} we deduce that
$|H| \leqs r^{2m}|{\rm Sp}_{2m}(r)|$, so $|H|^3<r^{6m^2+9m}<q^{6m^2+9m}$. Now, if $m \geqs 2$ then $6m^2+9m \leqs 3^{2m}-2$ and thus
$$|H|^3<q^{6m^2+9m} \leqs q^{r^{2m}-2} = q^{n^2-2}<|G|.$$
Now assume $m=1$, so $n=r$ and $|H| \leqs r^3(r^2-1)$. Note that $q \geqs 4$ since $r$ divides $q-1$. If $r \geqs 5$ then 
$$|H|^3 \leqs r^9(r^2-1)^3<4^{r^2-2} \leqs q^{n^2-2}<|G|.$$
Finally, suppose $(m,r)=(1,3)$, so $n=3$. Here
$$|H|^3 \leqs 3^9(3^2-1)^3<|{\rm L}_{3}(q)| = \frac{1}{3}q^3(q^2-1)(q^3-1)$$
for $q>8$, so it remains to deal with the cases $q=4,7$. In this situation, $H=3^2.Q_8$ (see \cite[Proposition 4.6.5]{KL}) and thus $H$ is large if and only if $q=4$.

Next suppose $H$ is of type $2^{2m}.{\rm Sp}_{2m}(2)$, so $n=2^m$, $m \geqs 2$ and $q=p \equiv 1 \imod{4}$. By \cite[Proposition 4.6.6]{KL} we have $|H| \leqs 2^{2m}|{\rm Sp}_{2m}(2)|<2^{2m^2+3m}$. Now, if $m \geqs 3$ then $6m^2+9m \leqs 2^{2m+1}-4$ so
$$|H|^3 < 2^{6m^2+9m} \leqs 2^{2^{2m+1}-4} = 4^{n^2-2}<q^{n^2-2}<|G|.$$
If $m=2$ then $n=4$, $|H| \leqs 2^4|{\rm Sp}_{4}(2)|$ and $|G|=\frac{1}{4}|{\rm SL}_{4}(q)|$, so $|H|^3<|G|$ if $q>5$. However, if $q=5$ then $H=2^4.A_6$ (see \cite[Proposition 4.6.6]{KL}) and this subgroup is large.

Finally, suppose $n=2$ and $H$ is of type $2^{1+2}_{-}.O_{2}^{-}(2)$, so $q=p \geqslant 3$. Then $H=A_4$ if $p \equiv \pm 3 \imod{8}$, otherwise $H=S_4$ (see \cite[Proposition 4.6.7]{KL}).  
Now, if $p \geqslant 31$ then $24^3<\frac{1}{2}p(p^2-1)$, so we may assume $p<31$. It is easy to check that the only large examples $(G,H)$ that arise in this situation are the following:
$$({\rm L}_{2}(23),S_4), \; ({\rm L}_{2}(17),S_4), \; ({\rm L}_{2}(13),A_4), \; ({\rm L}_{2}(11),A_4), \; ({\rm L}_{2}(7),S_4), \; ({\rm L}_{2}(5),A_4).$$
(We note that if $H=A_4$ then the maximality of $H$ implies that $p \not\equiv \pm 1\imod{10}$. In particular, $A_4<{\rm L}_{2}(11)$ is non-maximal.)
\end{proof}

\begin{lem}\label{l:c7}
The conclusion to Proposition \ref{p:psl} holds if $H \in \C_7$.
\end{lem}

\begin{proof}
Here $H$ is a tensor product subgroup of type ${\rm GL}_{a}(q) \otimes \cdots \otimes {\rm GL}_{a}(q)$ (with $t \geqslant 2$ factors), so $n=a^t$ and $a \geqs 3$ (see \cite[Table 3.5.A]{KL}). By \cite[Proposition 4.7.3]{KL} we have $|H|<|{\rm SL}_{a}(q)|^tt!$, and it is easy to check that  
$$(t!)^3 < 2^{3^{2t}-24t-3} < q^{a^{2t}-3t(a^2-1)-2}$$ 
for all $t \geqslant 2$ and $a \geqslant 3$. Therefore 
$$|H|^3<q^{3t(a^2-1)}(t!)^3<q^{a^{2t}-2} = q^{n^2-2} <|G|$$
and thus $H$ is non-large.
\end{proof}

\begin{lem}\label{l:c8}
The conclusion to Proposition \ref{p:psl} holds if $H \in \C_8$.
\end{lem}

\begin{proof}
There are three cases to consider. If $H$ is of type ${\rm Sp}_{n}(q)$ then $|H| \geqslant |{\rm PSp}_{n}(q)|>\frac{1}{4}q^{n(n+1)/2}=:\a$ (see \cite[Proposition 4.8.3]{KL}) and \eqref{e:al} holds. Similarly, if $H$ is of type $O_{n}^{\e}(q)$, where $n \geqslant 3$ and $q$ is odd, then $|H| = |{\rm SO}_{n}^{\e}(q)|>\frac{1}{2}q^{n(n-1)/2}=:\a$ (see \cite[Proposition 4.8.4]{KL}) and again we deduce that \eqref{e:al} holds, unless $n=3$ and $q \in \{3,5,7\}$. In each of these cases, $|H|=q(q^2-1)$ and it is easy to check that $|H|^3>|G|$. Finally, if $H$ is of type ${\rm U}_{n}(q_0)$, with $q=q_0^2$ and $n \geqslant 3$, then \cite[Proposition 4.8.5]{KL} implies that $|H| \geqs |{\rm U}_{n}(q_0)|>(1-q_0^{-1})q_0^{n^2-1}$ and we quickly deduce that $H$ is large.
\end{proof}

\subsubsection{Unitary groups}\label{sss:unitary}

\begin{prop}\label{p:psu}
Let $G={\rm U}_{n}(q)$ and let $H \in \C(G)$ be a geometric subgroup of $G$. Assume $n \geqs 3$ and set $d=(n,q+1)$. Then $H$ is large if and only if one of the following holds:
\begin{itemize}\addtolength{\itemsep}{0.2\baselineskip}
\item[{\rm (i)}] $H \in \C_1$;
\item[{\rm (ii)}] $H$ is a $\C_2$-subgroup of type ${\rm GU}_{n/t}(q)\wr S_t$ and one of the following holds:
\begin{itemize}\addtolength{\itemsep}{0.2\baselineskip}
\item[{\rm (a)}] $t=2$;
\item[{\rm (b)}] $t=3$ and either $q \in \{2,3,4\}$, or 
$$(q,d) \in \{ (5,3), (7,1), (7,2), (9,1), (9,2), (13,1),(16,1)\};$$
\item[{\rm (c)}] $4 \leqs n=t \leqs 11$ and either $q=2$, or 
$$(n,q) \in \{ (6,3),(5,3),(4,3),(4,4),(4,5)\}.$$
\end{itemize}
\item[{\rm (iii)}] $H$ is a $\C_2$-subgroup of type ${\rm GL}_{n/2}(q^2)$; 
\item[{\rm (iv)}] $H$ is a $\C_3$-subgroup of type ${\rm GU}_{n/k}(q^k)$, where $k=q=3$ and $n$ is odd;
\item[{\rm (v)}] $H$ is a $\C_5$-subgroup of type ${\rm GU}_{n}(q_0)$, where $q=q_0^3$ and $f>1$, where
\renewcommand{\arraystretch}{1.7}
$$f = \left\{\begin{array}{ll} 
\frac{27(n,q_0^3+1)}{(n,q_0+1)^3} & \mbox{if $(q_0^2-q_0+1)_3>1$ and $(q_0+1)_3 \geqslant n_3>1$} \\
\frac{(n,q_0^3+1)}{(n,q_0+1)^3} & \mbox{otherwise.}
\end{array}\right.$$
\renewcommand{\arraystretch}{1}
\item[{\rm (vi)}] $H$ is a $\C_5$-subgroup of type ${\rm Sp}_{n}(q)$ or $O_{n}^{\e}(q)$; 
\item[{\rm (vii)}] $H$ is a $\C_6$-subgroup and $(G,H)$ is one of the following:
$$({\rm U}_{4}(7), 2^4.{\rm Sp}_{4}(2)),\; ({\rm U}_{4}(3), 2^4.A_6),\; ({\rm U}_{3}(5), 3^2.Q_8).$$
\end{itemize}
\end{prop}

\begin{remk}\label{r:psu}
In part (vii) of Proposition \ref{p:psu}, the subgroup $3^2.Q_8<{\rm U}_{3}(5)$ is non-maximal.
\end{remk}

The proof of Proposition \ref{p:psu} is very similar to the proof of Proposition \ref{p:psl}, so for the sake of brevity we will only give details for subgroups in the $\C_2$, $\C_3$ and $\C_5$ collections. Set $d=(n,q+1)$ and recall that Corollary \ref{l:ord}(ii) gives 
$$(1-q^{-1})q^{n^2-2}<|G| \leqs (1-q^{-2})(1+q^{-3})q^{n^2-1}<q^{n^2-1}.$$
Note that $(n,q) \neq (3,2)$ since ${\rm U}_{3}(2) \cong 3^2.Q_8$ is not simple.

\begin{lem}\label{l:c2_u}
The conclusion to Proposition \ref{p:psu} holds if $H \in \C_2$.
\end{lem}

\begin{proof}
If $H$ is of type ${\rm GL}_{n/2}(q^2)$ then $|H|=2d^{-1}(q-1)|{\rm SL}_{n/2}(q^2)|$ (see \cite[Proposition 4.2.4]{KL}) and the bound $|H|^3>|G|$ quickly follows. For the remainder, let us assume $H$ is a $\C_2$-subgroup of type ${\rm GU}_{n/t}(q) \wr S_t$, so the order of $H$ can be read off from \cite[Proposition 4.2.9]{KL}. As in the proof of Lemma \ref{l:c2}, it is easy to check that $H$ is large if $t=2$. 

Next suppose $t \geqs 4$. We have
$$|H|<(1+q^{-1})^{t-1}q^{n^2/t-1}t!,\;\; |G|>(1-q^{-1})q^{n^2-2},$$
so $|H|^3<|G|$ if
\begin{equation}\label{e:uc3}
(t!)^3<(1-q^{-1})(1+q^{-1})^{3(1-t)}q^{n^2\left(1-\frac{3}{t}\right)+1}.
\end{equation}
If $n \geqs 2t$ then \eqref{e:uc3} holds if $t \geqs 5$ (set $n=2t$, $q=2$ and argue as in the proof of Lemma \ref{l:ord3}). Similarly, if $t=4$ and $n \geqs 8$ then \eqref{e:uc3} holds unless $(n,q) = (8,2)$; in this case it is easy to check the bound $|H|^3<|G|$ directly.

Now assume $n=t \geqs 4$. Here the reader can check that \eqref{e:uc3} holds unless $q=2$ and $n \leqs 11$, or $q=3$ and $n \leqs 6$, or $(n,q) = (5,4),(4,4),(4,5),(4,7)$ or $(4,8)$. In each of these cases, the desired result follows by directly computing $|H|$ and $|G|$.

To complete the proof, let us assume $t=3$. Here $H$ is non-large if and only if
\begin{equation}\label{e:psu2}
\frac{|{\rm GU}_{n/3}(q)|^9}{|{\rm GU}_{n}(q)|} < \frac{(q+1)^2d^2}{6^3}.
\end{equation}
For $n<9$ it is easy to check that $H$ is large if and only if
\begin{itemize}\addtolength{\itemsep}{0.2\baselineskip}
\item[(i)] $n=3$ and $q \in \{2,3,4,5,7,9,13,16\}$; or
\item[(ii)] $n=6$ and $q \in \{2,3,4,7,9,16\}$.
\end{itemize}
Now assume $n \geqs 9$. By applying Lemma \ref{l:ord4}(i) we have
$$f(q):=\frac{(1+q^{-1})^8(1-q^{-2})^8}{(1+q^{-3})}<\frac{|{\rm GU}_{n/3}(q)|^9}{|{\rm GU}_{n}(q)|} < (1+q^{-1})^8(1-q^{-2})^8(1+q^{-3})^9=:g(q)$$
and one can check that $g$ is a decreasing function. If $q \geqs 17$ then $g(q) \leqs g(17)< 50/27$ and thus $H$ is non-large since $(q+1)^2d^2/6^3 \geqs 50/27$ (minimal if $q=19$ and $d=1$).

If $q \in \{2,3,4\}$ then it is easy to check that $f(q)>(q+1)^2d^2/6^3$, so $H$ is large. Next suppose $q=5$, so $d=3$ or $6$. If $d=3$ then the value of $f(5)$ implies that $H$ is large. However, if $d=6$ then $g(5)<6 = (q+1)^2d^2/6^3$, so $H$ is non-large. The analysis of the remaining cases with $7 \leqs q \leqs 16$ is entirely similar.
\end{proof}

\begin{lem}\label{l:c3_u}
The conclusion to Proposition \ref{p:psu} holds if $H \in \C_3$.
\end{lem}

\begin{proof}
Here $H$ is of type ${\rm GU}_{n/k}(q^k)$, where $k$ is an odd prime. By \cite[Proposition 4.3.6]{KL} we have
$|H|=d^{-1}(q+1)^{-1}|{\rm GU}_{n/k}(q^k)|k$,
so $H$ is non-large if and only if
\begin{equation}\label{e:psu3}
\frac{|{\rm GU}_{n/k}(q^k)|^3}{|{\rm GU}_{n}(q)|} < \frac{(q+1)^2d^2}{k^3}.
\end{equation}
It is easy to check that this inequality holds if $k>3$, so let us assume $k=3$. We claim that $H$ is large if and only if $q=3$ and $n$ is odd.

First assume $n=3$. Here \eqref{e:psu3} holds if and only if
$$f(q):=\frac{(q^3+1)^2}{q^3(q+1)^3(q^2-1)}<\frac{(3,q+1)^2}{27}.$$
This inequality holds if $q=2$, but not if $q=3$. Since $f(q) \leqs f(4)< 1/27$
if $q \geqs 4$, we conclude that $H$ is large if and only if $q=3$. An entirely similar argument  shows that $H$ is non-large if $n=6$. Now assume $n \geqs 9$. Using Lemma \ref{l:ord4}(i) we calculate that 
$$\frac{|{\rm GU}_{n/3}(q^3)|^3}{|{\rm GU}_{n}(q)|} \leqs \frac{(1+q^{-3})^3(1-q^{-6})^3(1+q^{-9})^3}{(1+q^{-1})(1-q^{-2})}=:g(q).$$
If $q \geqs 4$ then $g(q)<1$, so \eqref{e:psu3} holds if $(q,d) \neq (4,1)$. In fact, $g(4)<9/10<25/27$, so \eqref{e:psu3} holds for all $q \geqs 4$. Similarly, if $q=2$ then $d=3$ and $g(2)<3$. Finally, suppose $q=3$. If $n$ is even then $d \geqs 2$ and $g(3)<1$, so $H$ is non-large. However, if $n$ is odd then $d=1$ and 
$$\frac{|{\rm GU}_{n/3}(3^3)|^3}{|{\rm GU}_{n}(3)|}>\frac{(1+3^{-3})^3(1-3^{-6})^3}{(1+3^{-1})(1-3^{-2})(1+3^{-3})}>\frac{16}{27}$$
so $H$ is large in this case.
\end{proof}

\begin{lem}\label{l:c5_u}
The conclusion to Proposition \ref{p:psu} holds if $H \in \C_5$.
\end{lem}

\begin{proof}
If $H$ is of type ${\rm Sp}_{n}(q)$ or $O_{n}^{\e}(q)$ then it is easy to check that $|H|^3>|G|$, so let us assume $H$ is a subfield subgroup of type ${\rm GU}_{n}(q_0)$, where $q=q_0^k$ and $k$ is an odd prime. We proceed as in the proof of Lemma \ref{l:c5}. According to \cite[Proposition 4.5.3]{KL} we have
$$|H| = (q_0+1)^{-1}\left(q_0+1, (q_0^k + 1)d^{-1}\right)|{\rm SU}_{n}(q_0)| \geqs |{\rm U}_{n}(q_0)|.$$
If $k \geqslant 5$ then $H$ is non-large since
$$|H|^3 \leqslant |{\rm SU}_{n}(q_0)|^3 <q^{\frac{3}{k}(n^2-1)}<\frac{1}{2}q^{n^2-2}<|G|.$$

Now assume $k=3$. This case requires some careful analysis. First observe that $H={\rm PSU}_{n}(q).e \leqs {\rm PGU}_{n}(q)$, where 
$$e = \left(q_0+1,\frac{q_0^3+1}{(n,q_0^3+1)}\right)(n,q_0+1)(q_0+1)^{-1}$$
and thus $|H|^3<|G|$ if and only if
$$\zeta := \frac{|{\rm GU}_{n}(q_0^3)|}{|{\rm GU}_{n}(q_0)|^3} \cdot \frac{(q_0+1)^3}{q_0^3+1}>e^3\frac{(n,q_0^3+1)}{(n,q_0+1)^3} =: f.$$ 
By applying \cite[Lemma B.6]{BG_book} we calculate that 
\renewcommand{\arraystretch}{1.7}
$$e = \left\{\begin{array}{ll} 
3 & \mbox{if $(q_0^2-q_0+1)_3>1$ and $(q_0+1)_3 \geqslant n_3>1$} \\
1 & \mbox{otherwise.}
\end{array}\right.$$
\renewcommand{\arraystretch}{1}

We claim that $H$ is large if and only if $f>1$. Using Lemma \ref{l:ord4}(i) we have
$$\zeta >\frac{(1-q_0^{-6})(q_0+1)^3}{(1+q_0^{-1})^3(1-q_0^{-2})^3(1+q_0^{-3})^2(q_0^3+1)} >1$$
for all $q_0 \geqslant 2$, whence $H$ is non-large if $f \leqslant 1$. It remains to show that $H$ is large if $f>1$. Again, using Lemma \ref{l:ord4}(i) we deduce that
$$\zeta < \frac{(1+q_0^{-3})(1-q_0^{-6})(1+q_0^{-9})(q_0+1)^3}{(1+q_0^{-1})^3(1-q_0^{-2})^3(q_0^3+1)} =:g(q_0),$$
so $H$ is large if $f \geqslant g(q_0)$. Note that $g$ is a decreasing function. 

If $q_0=2$ then the condition $f>1$ implies that $e=3$, so $n_3=3$, $f=3$ and it is easy to check that $3>g(2)$. Now suppose $q_0 \geqslant 3$. There are two cases to consider. First assume $e=3$. Then $(n,q_0+1) = 3\ell$ and $(n,q_0^3+1) = 3\ell m$ for some $\ell, m \geqs 1$, so $f = 3m/\ell^2$ and $q_0 \geqs 3\ell-1$. If $\ell=1$ then $f=3m$ and $f>g(3) \geqs g(q_0)>\zeta$, so $H$ is large. Now assume $\ell \geqs 2$. Since $f>1$ we deduce that 
\begin{equation}\label{e:pp2}
f \geqs \frac{\ell^2+1}{\ell^2} \geqs g(3\ell-1) \geqs g(q_0)>\zeta
\end{equation}
for all $\ell \geqs 1$, and thus $H$ is large. Finally, let us assume $e=1$, with $(n,q_0+1) = \ell$ and $(n,q_0^3+1) = \ell m$, so $f = m/\ell^2$ and $q_0 = a\ell-1$ for some $a \geqs 1$. Note that $m$ divides $q_0^2-q_0+1$. If $\ell=1$ then $f=m \geqs 2$ and $H$ is large since $f > g(3) \geqs g(q_0)>\zeta$. Now assume $\ell \geqs 2$. If $a \geqs 3$ then \eqref{e:pp2} holds. Similarly, if $a=2$ and $\ell \geqs 5$ then \eqref{e:pp2} holds (with 
$g(3\ell-1)$ replaced by $g(2\ell-1)$). If $a=2$ and $2\leqs \ell \leqs 4$ then it is easy to check that the condition $f>1$ implies that $f \geqs 7/4 > g(3)$. Finally, if $a=1$ then
$m \geqs \ell^2+1 = q_0^2+2q_0+2>q_0^2-q_0+1$, which is a contradiction, so this situation does not arise. We conclude that $H$ is large if $f>1$.
\end{proof}

\subsubsection{Symplectic groups}\label{sss:symplectic}

\begin{prop}\label{p:psp}
Let $G={\rm PSp}_{n}(q)'$ and let $H \in \C(G)$ be a geometric subgroup of $G$. Assume $n \geqs 4$. Then $H$ is large if and only if one of the following holds:
\begin{itemize}\addtolength{\itemsep}{0.2\baselineskip}
\item[{\rm (i)}] $H \in \C_1 \cup \C_8$;
\item[{\rm (ii)}] $H$ is a $\C_2$-subgroup of type ${\rm Sp}_{n/t}(q)\wr S_t$, where $t \leqs 3$, or $(n,t)=(8,4)$, or $(n,t) = (10,5)$ and $q=3$;
\item[{\rm (iii)}] $H$ is a $\C_2$-subgroup of type ${\rm GL}_{n/2}(q)$;
\item[{\rm (iv)}] $H$ is a $\C_3$-subgroup of type ${\rm Sp}_{n/2}(q^2)$, ${\rm Sp}_{n/3}(q^3)$ or ${\rm GU}_{n/2}(q)$; 
\item[{\rm (v)}] $H$ is a $\C_5$-subgroup of type ${\rm Sp}_{n}(q_0)$ with $q=q_0^2$;
\item[{\rm (vi)}] $H$ is a $\C_6$-subgroup and $(G,H)$ is one of the following:
$$({\rm PSp}_{8}(3), 2^6.\O_{6}^{-}(2)),\; ({\rm PSp}_{4}(7), 2^4.O_{4}^{-}(2)),\; ({\rm PSp}_{4}(5), 2^4.\O_{4}^{-}(2)),\; ({\rm PSp}_{4}(3), 2^4.\O_{4}^{-}(2)).$$ 
\end{itemize}
\end{prop}

The proof of this proposition is similar (and easier) to the proofs of Propositions \ref{p:psl} and \ref{p:psu}. For example, suppose $H$ is a $\C_2$-subgroup of type ${\rm Sp}_{n/3}(q) \wr S_3$. Then \cite[Proposition 4.2.10]{KL} gives $|H|=6d^{-1}|{\rm Sp}_{n/3}(q)|^3$, with $d=(2,q-1)$, and using Lemma \ref{l:ord4}(ii) we deduce that $H$ is large since
$$|H|^3> \left(\frac{3}{8}q^{\frac{1}{6}n^2+\frac{1}{2}n}\right)^3 > q^{\frac{1}{2}n(n+1)}>|G|.$$
Similarly, suppose $H$ is a $\C_5$-subgroup of type ${\rm Sp}_{n}(q_0)$ with $q=q_0^3$. Then \cite[Proposition 4.5.4]{KL} gives $H={\rm PSp}_{n}(q_0)$. In particular, if $q$ is odd then  
$$|H|^3< \left(\frac{1}{2}q_0^{\frac{1}{2}n(n+1)}\right)^3 = \frac{1}{8}q^{\frac{1}{2}n(n+1)}<|G|$$
and thus $H$ is non-large. By using the more accurate bounds in Lemma \ref{l:ord4}(ii), 
we see that the same conclusion holds when $q$ is even.

\subsubsection{Orthogonal groups}\label{sss:orthogonal}

\begin{prop}\label{p:pso}
Let $G={\rm P\O}_{n}^{\e}(q)$ and let $H \in \C(G)$ be a geometric subgroup of $G$. Assume $n \geqs 7$. Then $H$ is large if and only if one of the following holds:
\begin{itemize}\addtolength{\itemsep}{0.2\baselineskip}
\item[{\rm (i)}] $H \in \C_1$;
\item[{\rm (ii)}] $H$ is a $\C_2$-subgroup of type $O_{n/t}^{\e'}(q)\wr S_t$ and one of the following holds:
\begin{itemize}\addtolength{\itemsep}{0.2\baselineskip}
\item[{\rm (a)}] $t=2$;
\item[{\rm (b)}] $(n,t,q,\e,\e') = (12,3,2,-,-), (10,5,2,-,-)$ or $(8,4,2,+,-)$;
\item[{\rm (c)}] $n=t$ and either $(n,q) = (7,5)$, or $7 \leqs n \leqs 13$ and $q=3$.
\end{itemize}
\item[{\rm (iii)}] $H$ is a $\C_2$-subgroup of type ${\rm GL}_{n/2}(q)$;
\item[{\rm (iv)}] $H$ is a $\C_3$-subgroup of type $O_{n/2}^{\e'}(q^2)$ or ${\rm GU}_{n/2}(q)$; 
\item[{\rm (v)}] $H$ is a $\C_4$-subgroup of type ${\rm Sp}_{n/2}(q) \otimes {\rm Sp}_{2}(q)$  and $(n,\e) = (12,+)$ or $(8,+)$; 
\item[{\rm (vi)}] $H$ is a $\C_5$-subgroup of type $O_{n}^{\e'}(q_0)$ with $q=q_0^2$;
\item[{\rm (vii)}] $H$ is a $\C_6$-subgroup with $(G,H) = ({\rm P\O}_{8}^{+}(3), 2^6.\O_{6}^{+}(2))$;
\item[{\rm (viii)}] $H$ is a $\C_7$-subgroup with $G=\O_{8}^{+}(q)$, $H={\rm Sp}_{2}(q) \wr S_3$ and $q \leqs 2^7$ is even. 
\end{itemize}
\end{prop}

\begin{remk}\label{r:k}
By Kleidman \cite{K}, the $\C_6$-subgroup $2^6.\O_{6}^{+}(2)<{\rm P\O}_{8}^{+}(3)$ in case (vii), and also the $\C_2$-subgroup in (ii)(c) with $(n,\e,q)=(8,+,3)$, is non-maximal. Similarly, the $\C_7$-subgroup ${\rm Sp}_{2}(q) \wr S_3 < \O_{8}^{+}(q)$ in (viii) is also non-maximal.
\end{remk}

The proof of Proposition \ref{p:pso} is entirely similar to the proofs in the linear, unitary and symplectic cases. We give details for subgroups in the $\C_2$ and $\C_5$ collections. 

\begin{lem}\label{l:c2_o}
The conclusion to Proposition \ref{p:pso} holds if $H \in \C_2$.
\end{lem}

\begin{proof}
If $H$ is of type ${\rm GL}_{n/2}(q)$ (in which case $\e=+$ and $n/2$ is even) then \cite[Proposition 4.2.7]{KL} implies that $|H|>\frac{1}{8}q^{n^2/4}$, so $|H|^3>q^{n(n-1)/2}>|G|$ and thus $H$ is large. For the remainder, let us assume $H$ is of type $O_{n/t}^{\e'}(q) \wr S_t$ (or type $O_{n/2}(q)^2$ with $n/2$ odd). 

If $t=2$ then a combination of \cite[Propositions 4.2.11, 4.2.14, 4.2.16]{KL} yields $|H|>\frac{1}{4}q^{n(n-2)/4}$, and it quickly follows that $H$ is large. 

Next suppose $t=3$. If $n/3$ is odd (in which case $nq$ is odd and $n \geqs 9$) then $|H|<12q^{n(n-3)/6}$, and it is easy to check that $|H|^3<\frac{1}{4}q^{n(n-1)/2}<|G|$, so $H$ is non-large. Similarly, if $n/3$ is even then $|H|<12q^{n(n-3)/6}$, $|G|>\frac{1}{8}q^{n(n-1)/2}$ and we deduce that $H$ is non-large if $(n,q) \neq (12,2)$. In this exceptional case, we have $|H|=24|\O_{4}^{\e}(2)|^3$, and it is easy to check that $H$ is large if and only if $\e=-$. This case is recorded in part (ii)(b) of Proposition \ref{p:pso}.

To complete the proof, we may assume that $t \geqs 4$. First assume $n=t$, in which case $q=p \geqs 3$ (see \cite[Tables 3.5.D -- 3.5.F]{KL}). Then $|H|\leqs 2^{n-1}n!$ and $|G|>\frac{1}{8}q^{n(n-1)/2}$, so $H$ is non-large if $n!<2^{-n}q^{n(n-1)/6}$. It is straightforward to check that this bound holds unless $q=5$ and $n \leqs 8$, or $q=3$ and $n \leqs 14$. By using  \cite[Proposition 4.2.15]{KL} to compute the precise order of $H$, we obtain the list of large subgroups given in part (ii)(c) of Proposition \ref{p:pso}.

Finally, let us assume $t \geqs 4$ and $n \geqs 2t$. From \cite[Propositions 4.2.11 and 4.2.14]{KL} we deduce that $|H|<2^{t-1}q^{n(n/t-1)/2}t!$, so $|H|^3<|G|$ if
\begin{equation}\label{e:tf}
t!<2^{-t}q^{\frac{1}{6}n^2\left(1-\frac{3}{t}\right)+\frac{1}{3}n}.
\end{equation}
Since $n \geqs 2t$ and $q \geqs 2$, it follows that this inequality holds if $t!<2^{t(2t-7)/3}$, which is true if $t \geqs 6$. Finally, if $t=4$ or $5$ then \eqref{e:tf} holds unless $(n,q) = (10,2),(8,3)$ or $(8,2)$. If $(n,q)=(10,2)$ then $\e=-$, $H$ is of type $O_{2}^{-}(2) \wr S_5$ (see \cite[Tables 3.5.E and 3.5.F]{KL}) and it is easy to check that $H$ is large. Similarly, if $n=8$ and $q \leqs 3$ then $\e=+$, $H$ is of type $O_{2}^{\pm}(q) \wr S_4$ and we find that $H$ is large if and only if $G=\O_{8}^{+}(2)$ and $H$ is of type $O_{2}^{-}(2) \wr S_4$. We record these cases in part (ii)(b) of Proposition \ref{p:pso}. 
\end{proof}

\begin{lem}\label{l:c5_o}
The conclusion to Proposition \ref{p:pso} holds if $H \in \C_5$.
\end{lem}

\begin{proof}
Let $H$ be a $\C_5$-subgroup of type $O_{n}^{\e'}(q_0)$, where $q=q_0^k$ for a prime $k$. If $k=2$ then $|H|>\frac{1}{4}q^{n(n-1)/4}$ (see \cite[Propositions 4.5.8, 4.5.10]{KL}) and thus $|H|^3>q^{n(n-1)/2}>|G|$. Similarly, if $k \geqs 5$ then $|H|<q^{n(n-1)/10}$ and thus $H$ is non-large since $|H|^3<\frac{1}{8}q^{n(n-1)/2}<|G|$. 

To complete the proof, we may assume that $k=3$. Suppose $nq$ is odd. Then
$$|H|=|\O_{n}(q_0)| <\frac{1}{2}q_0^{n(n-1)/2}(1-q_0^{-2})(1-q_0^{-4}),\;\; |G|>\frac{1}{2}q^{n(n-1)/2}(1-q^{-2}-q^{-4})$$
(see Lemma \ref{l:ord4}(iii)) and we quickly deduce that $H$ is non-large. The remaining cases are similar. For example, suppose $n$ is even, $\e=+$ and $q$ is odd. Then \cite[Proposition 4.5.10]{KL} implies that
$$|H|<\frac{1}{2}q_0^{n(n-1)/2}(1-q_0^{-2})(1-q_0^{-4}),\;\; |G|>\frac{1}{4}q^{n(n-1)/2}(1-q^{-2}-q^{-4})(1-q^{-n/2})$$
and once again we deduce that $H$ is non-large.
\end{proof}

\subsection{Almost simple irreducible subgroups}\label{ss:nongeom}

Let $G$ be a finite simple classical group over $\F$ with natural module $V$. Write $q=p^a$ with $p$ a prime, and set $n=\dim V$. Recall that Aschbacher's main theorem in \cite{asch} states that if $H$ is a maximal subgroup of $G$ then either $H$ belongs to one of the geometric $\C_i$ collections, or $H$ is almost simple with socle $H_0$, and the full covering group of $H_0$ acts absolutely irreducibly on $V$. We write $\mathcal{S}(G)$ to denote the latter collection of subgroups, and in this section we determine the large subgroups in $\mathcal{S}(G)$. The following definition, taken from \cite[p.3]{KL}, gives a more detailed description of the subgroups in $\mathcal{S}(G)$ (note that conditions (ii) -- (vii) ensure that a subgroup in $\mathcal{S}(G)$ is not contained in one of the geometric $\C_i$ collections).

\begin{defn}\label{sdef}
Let $G$ be a finite simple classical group over $\F$ with natural module $V$.
A subgroup $H$ of $G$ is in the collection $\mathcal{S}(G)$ if and only if the following hold:
\begin{itemize}\addtolength{\itemsep}{0.2\baselineskip}
\item[(i)] The socle $H_{0}$ of $H$ is a nonabelian simple group, and $H_0 \not\cong G$.
\item[(ii)] If $\what{H}_{0}$ is the full covering group of $H_{0}$, and if $\rho : \what{H}_{0} \to {\rm GL}(V)$ is a representation of $\what{H}_{0}$ such that $\rho(\what{H}_{0})=H_{0}$ (modulo scalars) then $\rho$ is absolutely irreducible.
\item[(iii)] $\rho(\what{H}_{0})$ cannot be realized over a proper subfield of $\mathbb{F}$, where $\mathbb{F}=\mathbb{F}_{q^{2}}$ if  $G$ is unitary, otherwise $\mathbb{F}=\F$.
\item[(iv)] If $\rho(\what{H}_{0})$ fixes a non-degenerate quadratic form on $V$ then $G={\rm P\O}_{n}^{\e}(q)$.
\item[(v)] If $\rho(\what{H}_{0})$ fixes a non-degenerate symplectic form on $V$, but no non-degenerate quadratic form, then $G={\rm PSp}_{n}(q)$.
\item[(vi)] If $\rho(\what{H}_{0})$ fixes a non-degenerate unitary form on $V$ then $G={\rm U}_{n}(q)$.
\item[(vii)] If $\rho(\what{H}_{0})$ does not satisfy the conditions in (iv), (v) or (vi) then $G={\rm L}_{n}(q)$.
\end{itemize}
\end{defn}

Our main result is the following:

\begin{prop}\label{p:scoll}
Let $G$ be a finite simple classical group, and let $H$ be a maximal subgroup of $G$ in the collection $\mathcal{S}(G)$. Then $H$ is large if and only if $(G,H)$ is one of the cases recorded in Table \ref{tab:sg}.
\end{prop}

\begin{table}
$$\begin{array}{lll} \hline
G & H &  \mbox{Conditions} \\ \hline
\O_{22}^{+}(2) &  A_{24}   & \\
{\rm Sp}_{20}(2) & S_{22}  & \\
\O_{20}^{-}(2) & A_{21}  & \\
\O_{18}^{-}(2) & A_{20}  & \\
{\rm Sp}_{16}(2) & S_{18}  & \\
\O_{16}^{+}(2) & A_{17}  & \\
\O_{14}^{+}(2) & A_{16}  & \\
{\rm Sp}_{12}(2) & S_{14}  & \\
\O_{12}^{-}(2) & A_{13}  & \\
{\rm P\O}_{10}^{+}(3) & A_{12}  & \\
\O_{10}^{-}(2) & A_{12}  & \\
\O_9(3) & A_{10}  & \\
{\rm Sp}_{8}(2) & S_{10}  & \\
{\rm P\O}_{8}^{+}(q) & \O_7(q)  & p \neq 2 \\
& {\rm Sp}_{6}(q)  & p=2 \\
& {\rm P\O}_{8}^{-}(q_0)  & q=q_0^2 \\
& {}^3D_4(q_0)  & q=q_0^3, \, p \neq 2 \\
& \O_{8}^{+}(2)  & q=3,5,7 \\
 & A_{9}  & q=2 \\
\O_7(q) & G_2(q)  &  \\
& {\rm Sp}_{6}(2)  & q=3,5,7 \\
 & S_9  & q=3 \\
{\rm U}_{6}(q) & {\rm U}_{4}(3).2  & q=2 \\
& {\rm M}_{22}  & q=2 \\
{\rm PSp}_{6}(q) & G_2(q)  & p = 2 \\
& {\rm U}_{3}(3).2  & q=2 \\
& {\rm J}_{2}  & q=5 \\
{\rm L}_{5}(q) & {\rm M}_{11}  & q=3 \\
{\rm U}_{5}(q) & {\rm L}_{2}(11)  & q=2 \\
{\rm L}_{4}(q) & A_7  & q=2 \\
& {\rm U}_{4}(2)  & q=7 \\
{\rm U}_{4}(q) & A_7  & q=3,5 \\
& {\rm L}_{3}(4)  & q=3 \\
& {\rm U}_{4}(2)  & q=5 \\
{\rm PSp}_{4}(q)' & {}^2B_2(q)  & \mbox{$p=2$, $\log_2q>1$ odd} \\ 
& A_7 & q= 7 \\
& A_6  & q=5 \\
& A_5  & q=2 \\
{\rm L}_{3}(q) & A_6  & q= 4 \\
{\rm U}_{3}(q) & A_7  & q=5 \\
& {\rm M}_{10}  & q=5 \\
& {\rm L}_{2}(7)  & q=3 \\
{\rm L}_{2}(q) & A_5  & \mbox{See Remark \ref{r:cond}} \\ \hline 
\end{array}$$
\caption{The large maximal subgroups in $\mathcal{S}(G)$}
\label{tab:sg}
\end{table}

\begin{remk}\label{r:cond}
For the case $(G,H) = ({\rm L}_{2}(q),A_5)$ in Table \ref{tab:sg}, we require 
$$q \in \{9,11,19,29,31,41,49,59,61,71\}.$$
\end{remk}

We start by recalling a special case of the main theorem of \cite{L}. The subcollection 
$\mathcal{A} \subseteq \mathcal{S}(G)$ is defined in Table \ref{atab}; here $H_0=A_d$ is an alternating group and $V$ is the fully deleted permutation module for $A_d$ over $\mathbb{F}_{p}$ (see \cite[pp.185--187]{KL}). Also recall the lower bound
\begin{equation}\label{e:lbd}
|G|>\frac{1}{8}q^{\frac{1}{2}n(n-1)}
\end{equation}
given in Corollary \ref{l:ord}.

\begin{table}
$$\begin{array}{lll} \hline
 d & p & \hspace{5mm} G \\ \hline
 d \equiv 2 \imod{4} & 2 & \hspace{5mm} {\rm Sp}_{d-2}(2)\\
 d \equiv 0 \imod{4} & 2 & \left\{ \begin{array}{ll} \O_{d-2}^{+}(2) & \hspace{2.8mm} \mbox{if $d \equiv 0 \imod{8}$}\\ \O_{d-2}^{-}(2) & \hspace{2.8mm} \mbox{if $d \equiv 4 \imod{8}$} \end{array} \right. \\
 \mbox{odd} & 2 & \left\{ \begin{array}{ll} \O_{d-1}^{+}(2) & \hspace{2.8mm} \mbox{if $d \equiv \pm 1 \imod{8}$}\\ 
\O_{d-1}^{-}(2) & \hspace{2.8mm} \mbox{if $d \equiv \pm 3 \imod{8}$} 
\end{array}\right. \\ 
\mbox{arbitrary} & \mbox{odd} & \left\{ \begin{array}{ll}
{\rm P\O}_{d-1}^{\e}(p) & \mbox{if $(d,p)=1$}\\
{\rm P\O}_{d-2}^{\e}(p) & \mbox{otherwise} \end{array} \right. \\ \hline
\end{array}$$
\caption{The collection $\mathcal{A}$ from Theorem \ref{lieb} (with $H_{0}=A_{d}$)}
\label{atab}
\end{table}

\begin{thm}\label{lieb}
If $H \in \mathcal{S}(G) \setminus \mathcal{A}$ then $|H|<q^{3un}$, where $u=2$ if $G = {\rm U}_{n}(q)$, otherwise $u=1$.
\end{thm}

\begin{proof}
This is a special case of \cite[Theorem 4.1]{L}.
\end{proof}

\begin{cor}\label{c:1}
If $n \geqslant 23$ then there are no large subgroups in $\mathcal{S}(G)$.
\end{cor}

\begin{proof}
First assume $H \in \mathcal{S}(G) \setminus \mathcal{A}$. If $G = {\rm U}_{n}(q)$ then Theorem \ref{lieb} implies that $|H|<q^{6n}$ and it is easy to check that $|H|^3<(1-q^{-1})q^{n^2-2}<|G|$ if $n \geqslant 20$. Similarly, if $G \neq {\rm U}_{n}(q)$ then $|H|<q^{3n}$ and by combining this bound with \eqref{e:lbd} we deduce that $|H|^3<|G|$ if $n \geqslant 20$. Now assume $H \in \mathcal{A}$, so $H_0 = A_d$ and $n \in \{d-2,d-1\}$. If $p$ is odd then it is straightforward to check that 
\begin{equation}\label{e:1}
|H|^3 \leqslant |S_d|^3 = (d!)^3< \frac{1}{8}p^{\frac{1}{2}(d-2)(d-3)}<|G|
\end{equation}
for all $d \geqslant 16$, so we may assume $p=2$. If $d \equiv 0 \imod{4}$ then $n=d-2$  and $(d!)^3<2^{(d-2)(d-3)/2-1}$ for all $d \geqslant 28$. Similarly, if $d \not\equiv 0 \imod{4}$ then $|G|>2^{(d-1)(d-2)/2-1}$ (note that $G={\rm Sp}_{d-2}(2)$ if $d \equiv 2 \imod{4}$) and the result quickly follows. 
\end{proof}

For the remainder we may assume that $n \leqslant 22$. First we classify the large  maximal subgroups in the subcollection $\mathcal{A}$.

\begin{prop}\label{p:ad}
Suppose $H \in \mathcal{A}$ is a maximal subgroup of $G$. Then $H$ is large if and only if $(G,H)$ is one of the cases recorded in Table \ref{tab:ad}.
\end{prop}

\begin{table}
$$\begin{array}{ll} \hline
 d & (G,H) \\ \hline
 5 & ({\rm P\O}_{4}^{-}(3), A_5), \;  ({\rm P\O}_{4}^{-}(7), A_5) \\
 6 & (\O_5(5),A_6) \\
 7 & (\O_5(7),A_7),\; (\O_{6}^{+}(2),A_7),\; ({\rm P\O}_{6}^{-}(3),A_7),\; ({\rm P\O}_{6}^{-}(5),A_7) \\
 9 & (\O_7(3),S_9),\; (\O_{8}^{+}(2),A_9) \\
 10 & ({\rm Sp}_{8}(2),S_{10}),\; (\O_9(3),A_{10}) \\
 12 & (\O_{10}^{-}(2),A_{12}),\; ({\rm P\O}_{10}^{+}(3),A_{12}) \\
 13 &  (\O_{12}^{-}(2),A_{13}) \\
 14 & ({\rm Sp}_{12}(2),S_{14}) \\
 16 & (\O_{14}^{+}(2),A_{16}) \\
 17 & (\O_{16}^{+}(2),A_{17}) \\
 18 & ({\rm Sp}_{16}(2), S_{18}) \\
 20 & (\O_{18}^{-}(2),A_{20}) \\
 21 & (\O_{20}^{-}(2),A_{21}) \\
 22 & ({\rm Sp}_{20}(2),S_{22}) \\
 24 & (\O_{22}^{+}(2),A_{24}) \\ \hline
\end{array}$$
\caption{The large maximal subgroups in $\mathcal{A}$, $H_0=A_d$}
\label{tab:ad}
\end{table}

\begin{proof}
Here $H_0=A_d$, and we will start by assuming $p$ is odd. First assume $d=5$. If $p=5$ then $G = \O_3(5) \cong A_5$, so we may assume that $p \neq 5$ and $G={\rm P\O}_{4}^{-}(p) \cong {\rm L}_{2}(p^2)$, in which case $H=A_5$ and $p \equiv \pm 3 \imod{10}$ (see \cite[Table 8.2]{BHR}, for example). It is easy to check that $H$ is large if and only if $p=3$ or $7$. Next suppose $d=6$. The case $p=3$ can be discarded since ${\rm P\O}_{4}^{-}(3) \cong A_6$. If $p=5$ then $A_6<\O_5(5)$ is large, and we note that $H$ is non-maximal if $p=7$. If $p \geqslant 11$ then $|{\rm Aut}(A_6)|^3<|G|$. The case $d=7$ is similar. Here the only large examples are $A_7<\O_5(7)$, $A_7<{\rm P\O}_{6}^{-}(3)$ and $A_7<{\rm P\O}_{6}^{-}(5)$.

Next suppose $p$ is odd and $d \geqslant 8$. Here \eqref{e:1} holds unless $(p,d) = (7,8)$, or $p=5$ and $8 \leqslant d \leqslant 10$, or $p=3$ and $8 \leqslant d \leqslant 15$. If $d=8$ and $p=5$ or $7$ then $G=\O_7(p)$ and $(8!)^3<|G|$. Similarly, if $(p,d)=(5,9)$ then $|G| \geqslant |{\rm P\O}_{8}^{+}(5)|>(9!)^3$, and if $(d,p)=(10,5)$ we have $G={\rm P\O}_{8}^{+}(5)$ and $H=A_{10}$ (see \cite[Section 8.2]{BHR}),
so $|H|^3<|G|$. 

To complete the analysis when $p$ is odd, let us assume $p=3$. If $d=8$ then $G=\O_7(3)$ and $H$ is non-maximal; indeed, $S_9$ is a large maximal subgroup of $\O_7(3)$. Similarly, if $d=10$ then $H=A_{10}$ is a large maximal subgroup of $G=\O_9(3)$.
If $d=11$ then $H$ is non-maximal in $G={\rm P\O}_{10}^{+}(3)$, but $A_{12}$ is a large maximal subgroup of ${\rm P\O}_{10}^{+}(3)$ (see \cite[Table 8.67]{BHR}). Finally, if $d \geqslant 13$ then it is easy to check that $(d!)^3 <|G|$.

For the remainder of the proof, let us assume $p=2$. If $d \equiv 2 \imod{4}$ then $d \geqslant 10$ (since $S_6 \cong {\rm Sp}_{4}(2)$), $H=S_{d}$ and $G={\rm Sp}_{d-2}(2)$, so $|G|>2^{(d-1)(d-2)/2-1}$. It follows that $H$ is large if and only if $d=10,14,18$ or $22$. Similarly, if $d \equiv 0 \imod{4}$ then $d \geqslant 12$ (since $A_8 \cong \O_{6}^{+}(2)$), $H=A_{d}$ and $G=\O_{d-2}^{\e}(2)$, so $|G|>2^{(d-2)(d-3)/2-1}$ and $H$ is large if and only if $d=12,16,20$ or $24$. Finally, suppose $p=2$ and $d$ is odd, so $H=A_{d}$, $d \geqslant 7$ (since $A_5 \cong \O_{4}^{-}(2)$) and $|G|>2^{(d-1)(d-2)/2-1}$. It follows that $H$ is large if and only if $7 \leqslant d \leqslant 21$ (with $d$ odd), and we note that $H$ is non-maximal if 
$d=11,15$ or $19$.
\end{proof}

\begin{remk}\label{r:s}
Note that in the first three rows of Table \ref{tab:ad} we have $G={\rm P\O}_{n}^{\e}(q)$ with $n<7$. In each of these cases, in Table \ref{tab:sg}, we replace $G$ by an isomorphic linear, unitary or symplectic group. For example, the case $(G,H) = ({\rm P\O}_{4}^{-}(3),A_5)$ is recorded as $({\rm L}_{2}(9), A_5)$.
\end{remk}

We now partition the remaining subgroups in $\mathcal{S}(G)\setminus \mathcal{A}$ into two subcollections. Let ${\rm Lie}(p)$ be the set of simple groups of Lie type defined over a field  of characteristic $p$. Let $\mathcal{B}$ be the subgroups $H \in \mathcal{S}(G)\setminus \mathcal{A}$ such that $H_0 \not\in {\rm Lie}(p)$, and let $\mathcal{C}$ be the remaining subgroups in $\mathcal{S}(G)\setminus \mathcal{A}$. Note that in the next proposition we avoid repeating any cases that have already been included in Table \ref{tab:ad}. For example, the large subgroup ${\rm L}_{2}(9)<{\rm PSp}_{4}(5)$ is recorded in Table \ref{tab:ad} as $A_6<\O_5(5)$.

\begin{prop}\label{p:b}
Suppose $H \in \mathcal{B}$ is a maximal subgroup of $G$. Then $H$ is large if and only if $(G,H)$ is one of the cases recorded in Table \ref{tab:b}.
\end{prop}

\begin{table}
$$\begin{array}{lll} \hline
G & H & \mbox{Conditions} \\ \hline
{\rm P\O}_{8}^{+}(q) & \O_8^{+}(2) & q=3,5,7 \\
\O_7(q) & {\rm Sp}_{6}(2) & q=3,5,7 \\
{\rm U}_{6}(q) & {\rm U}_{4}(3).2 & q=2 \\
& {\rm M}_{22} & q=2 \\
{\rm PSp}_{6}(q) & {\rm U}_{3}(3).2 & q=2 \\
& {\rm J}_{2} & q=5 \\
{\rm L}_{5}(q) & {\rm M}_{11} & q=3 \\
{\rm U}_{5}(q) & {\rm L}_{2}(11) & q=2 \\
{\rm L}_{4}(q) & A_7 & q=2 \\
& {\rm U}_{4}(2) & q=7 \\
{\rm U}_{4}(q) & A_7 & q=3,5 \\
& {\rm L}_{3}(4) & q=3 \\
& {\rm U}_{4}(2) & q=5 \\
{\rm PSp}_{4}(q)' 
& A_5 & q=2 \\
{\rm L}_{3}(q) & A_6 & q= 4 \\
{\rm U}_{3}(q) & A_7 & q=5 \\
& {\rm M}_{10} & q=5 \\
& {\rm L}_{2}(7) & q=3 \\
{\rm L}_{2}(q) & A_5 & \mbox{See Remark \ref{r:cond}} \\ \hline 
\end{array}$$
\caption{The large maximal subgroups in $\mathcal{B}$}
\label{tab:b}
\end{table}

\begin{proof}
In view of the proof of Corollary \ref{c:1}, we may assume that $n \leqslant 19$. To begin with, let us assume $H_0 \not\cong {\rm L}_{2}(t)$ for any $t$, in which case the possibilities for $(G,H_0)$ can be read off from the information provided in \cite[Table 2]{HM} and the tables in \cite[Section 8.2]{BHR} (for $n \leqs 12$).

If $11 \leqslant n \leqslant 19$ then it is easy to check that no large examples arise. For example, suppose $n=11$. By inspecting \cite{HM}, we deduce that 
$$H_0 \in \{{\rm L}_{3}(3), {\rm U}_{5}(2), {\rm M}_{11}, {\rm M}_{12}, {\rm M}_{23}, {\rm M}_{24}\}.$$
Therefore $|H| \leqslant |{\rm M}_{24}|$ and by applying the lower bound in \eqref{e:lbd} we deduce that $|H|^3<|G|$ if $q \neq 2$. Finally, if $q=2$ then $G={\rm L}_{11}(2)$ and $H={\rm M}_{24}$ (see \cite{BHR}), whence 
$|H|^3<|G|$. 

Next suppose $n=10$.  By inspecting \cite{HM} (or \cite{BHR}), we deduce that $|H| \leqslant 2|{\rm U}_{5}(2)|$, so \eqref{e:lbd} implies that $|H|^3<|G|$ if $q>3$. Now assume $q=2$ or $3$. If $H_0={\rm U}_{5}(2)$ then $G={\rm PSp}_{10}(3)$ and thus $|H|^3<|G|$. The other cases with $n=10$ are handled similarly. Note that if $H_0={\rm M}_{12}$ and $q=2$ then $G=\O_{10}^{-}(2)$ and $H$ is non-maximal (indeed,  
${\rm M}_{12}<A_{12}<\O_{10}^{-}(2)$). Similarly, if $(n,q)=(9,2)$ and $H_0={\rm J}_{3}$ then $G={\rm U}_{9}(2)$ and $H={\rm J}_{3}$ is non-large. Finally, if $n \leqslant 8$ then the possibilities for $H$ can be read off from the relevant tables in \cite{BHR}, and the desired result quickly follows. 

To complete the proof, let us assume $H_0 = {\rm L}_{2}(t)$, where $t$ is an $r$-power for some prime $r \neq p$. The various cases that arise are recorded in \cite[Table 2]{HM2} (see cases (b) -- (d) therein), which includes the Frobenius-Schur indicator of the corresponding representation of $\what{H}_{0}$, and information on the field size $t$. We will give details for the case $t\equiv 1\imod{4}$ (corresponding to case (b) in \cite[Table 2]{HM2}); the other cases are very similar.

Suppose $t\equiv 1\imod{4}$. There are several cases to consider:
$$\begin{array}{ll}
G = {\rm PSp}_{n}(q)': & n \in \{(t - 1)/2,  t \pm 1\} \\
G = {\rm P\O}_{n}^{\e}(q): &  n \in \{ (t+1)/2, t \pm 1, t \} 
\end{array}$$

First assume $n=(t-1)/2$, so $G={\rm PSp}_{n}(q)'$ and
\begin{equation}\label{e:hbd}
|H| \leqs |{\rm Aut}({\rm L}_{2}(t))| = \log_{r}t\cdot t(t^2-1).
\end{equation}
Since $|G|>\frac{1}{4}q^{n(n+1)/2}$, we quickly deduce that $|H|^3<|G|$ if $t>17$. If $t=17$ then $n=8$, and the above bounds imply that $H$ is non-large if $q>2$. Similarly, if $q=2$ then $G={\rm PSp}_{8}(2)$ and $H={\rm L}_{2}(17)$ is non-large. For $t=13$ we have $n=6$ and we are left with the cases $q=2,3$. If $q=2$ then $H$ is non-maximal, and if $q=3$ we calculate that $H$ is non-large since $G={\rm PSp}_{6}(3)$ and $H={\rm L}_{2}(13)$ (see \cite[Table 8.29]{BHR}). Next assume $t=9$, so $H_0 \cong A_6$, $n=4$ and the usual bounds are sufficient if $q>9$. The cases with $q \leqs 9$ can be checked directly, and we find that ${\rm L}_{2}(9)<{\rm PSp}_{4}(5)$ is the only large subgroup (this corresponds to the $\mathcal{A}$ collection subgroup $A_6 < \O_5(5)$ in Table \ref{tab:ad}). Note that $H$ is non-maximal if $q=7$, and the case $q=2$ is excluded since ${\rm L}_{2}(9) \cong {\rm PSp}_{4}(2)'$. Finally, suppose $t=5$, so $G={\rm L}_{2}(q)$ and $H \cong A_5$, with $q \geqs 7$. It is easy to check that $H$ is large if and only if $q \leqs 73$ (and $q \neq 64$). In fact, $H$ is maximal in $G$ if and only if $q=p$ and $q \equiv \pm 1\imod{10}$, or if $q=p^2$ and $p \equiv \pm 3 \imod{10}$, which explains the particular values of $q$ listed in Remark \ref{r:cond}.

The other cases with $t\equiv 1\imod{4}$ are very similar. For example, suppose $n=(t+1)/2$, so $t \geqs 5$, $nq$ is odd and $G=\O_{n}(q)$.  Using \eqref{e:hbd} and the lower bound $|G|>\frac{1}{4}q^{n(n-1)/2}$ we deduce that $H$ is non-large if $t \geqs 17$. If $t=13$ then it remains to deal with the case $q=3$. Here $G=\O_7(3)$ and $H = {\rm L}_{2}(13)$ is non-maximal. Similarly, if $t=9$ then the previous bounds imply that $H$ is non-large if $q>9$; if $q=7$ or $9$ then $H$ is non-maximal, and if $q=5$ we have $G=\O_5(5)$ and $H={\rm L}_{2}(9)$ is large (this is an $\mathcal{A}$ collection subgroup $A_6<\O_5(5)$). Finally, if $t=5$ then $G=\O_3(q) \cong {\rm L}_{2}(q)$ and $H={\rm L}_{2}(5) \cong A_5$, so as before we calculate that $H$ is large if and only if $q \leqs 73$.
\end{proof}

\begin{prop}\label{p:c}
Suppose $H \in \mathcal{C}$ is a maximal subgroup of $G$. Then $H$ is large if and only if $(G,H)$ is one of the cases recorded in Table \ref{tab:c}.
\end{prop}

\begin{table}
$$\begin{array}{lll} \hline
G & H & \mbox{Conditions} \\ \hline
{\rm P\O}_{8}^{+}(q) & \O_{7}(q) & p \neq 2  \\
 & {\rm Sp}_{6}(q) & p = 2 \\
 & {\rm P\O}_{8}^{-}(q_0) & q=q_0^2 \\
 & {}^3D_4(q_0) & q=q_0^3,\, p \neq 2 \\
\O_7(q) & G_2(q) & p \neq 2 \\
{\rm Sp}_{6}(q) & G_2(q) & p = 2 \\
{\rm PSp}_{4}(q)' & {}^2B_2(q) & \mbox{$p=2$, $\log_2q>1$ odd} \\ \hline 
\end{array}$$
\caption{The large maximal subgroups in $\mathcal{C}$}
\label{tab:c}
\end{table}

\begin{proof}
As in the proof of the previous proposition, we may assume that $n \leqslant 19$. We consider each possibility for $H_0$ in turn; the analysis is similar in each case. 

To illustrate the general approach, consider the case $H_0 = {\rm L}_{m}(q^i)$, where $m \geqslant 2$ and $i \geqslant 1$. Recall that $V$ is not the natural $H_0$-module, nor its dual. Therefore $n \geqslant \frac{1}{2}m(m-1)$ (see \cite{Lu}, for example) so we may assume that $m \leqslant 6$ (since $n \leqslant 19$). If $i \geqslant 2$ then \cite[Proposition 5.4.6(i)]{KL} implies that $n = k^i \geqslant m^i$, where $k$ is the dimension of an irreducible ${\rm SL}_m(K)$-module (here $K$ denotes the algebraic closure of $\mathbb{F}_{q}$), so $m \leqslant 4$ and it is easy to check that no large examples arise. For example, if $m=3$ then the only possibility is $(k,i)=(3,2)$, so $G = {\rm L}_{9}(q)$ (since the relevant representation of $\what{H}_0$ is not self-dual) and clearly $|{\rm Aut}({\rm L}_{3}(q^2))|^3<|G|$. 

Now assume $i=1$. The possibilities for $n$ can be read off from the relevant tables in \cite{Lu}, and again we quickly deduce that there are no large examples. For example, suppose $m=4$, in which case $n \in \{6,10,14,15,16,19\}$ (see \cite[Table A.7]{Lu}). We can immediately rule out the case $n=6$ since $G={\rm P\O}_{6}^{+}(q)$ and the corresponding representation of $\what{H}_0$ induces an isomorphism $H_0 \cong G$, so this does not correspond to a subgroup in $\mathcal{S}(G)$. If $n=10$ then $G={\rm L}_{10}(q)$ (the representation is not self-dual, so $G$ is neither symplectic nor orthogonal), whence $|H|^3<|G|$. Finally, if $n \geqslant 14$ then the lower bound on $|G|$ given in \eqref{e:lbd} is good enough to show that $H$ is non-large. The remaining cases with $m \leqslant 6$ are similar. Note that ${\rm L}_{3}(q).2<\O_7(q)$ is non-maximal if $p=3$ (since ${\rm L}_{3}(q).2<G_2(q)<\O_7(q)$).

The other cases are very similar and we leave the reader to check the details. Note that the maximal subgroup ${}^3D_4(q_0)<{\rm P\O}_{8}^{+}(q)$ (where $q=q_0^3$) is large if and only if $q$ is odd.
\end{proof}

\section{Finite exceptional groups}\label{s:ex}

In this section we will assume $G$ is a finite simple group of exceptional Lie type; our aim is to prove Theorem \ref{t:main4}. Note that the
order of $G$ is given in \cite[Table 5.1.B]{KL}. 

It will be convenient to adopt the Lie notation for groups of Lie type, so for example, we will write $A_1(q)$ in place of ${\rm L}_{2}(q)$, and so on. Also recall that the \emph{type} of a subgroup $H$ of $G$ provides an approximate description of the group-theoretic structure of $H$. As before, we write ${\rm Lie}(p)$ for the set of finite simple groups of Lie type defined over fields of characteristic $p$. Finally, if $L$ is a group of Lie type then the \emph{untwisted Lie rank} of $L$, denoted ${\rm rk}(L)$, is the rank of the ambient algebraic group. For instance, $A_2(q) = {\rm L}_{3}(q)$ and $A_2^{-}(q) = {\rm U}_{3}(q)$ both have untwisted Lie rank $2$.

In this section, for integers $c \geqs 2$ and $d \geqs 3$, let $c_d$ be the largest \emph{primitive prime divisor} of $c^d-1$. That is, $c_d$ is the largest prime number with the property that $c_d$ divides $c^d-1$, and $c_d$ does not divide $c^i-1$ if $i<d$. By Zsigmondy \cite{Zsig}, such a prime $c_d$ exists unless $(c,d)=(2,6)$.

\begin{prop}\label{p:ex1}
The conclusion to Theorem \ref{t:main4} holds when $G$ is one of the following:
$$E_{6}^{\e}(2),\, F_4(2), \, {}^2F_4(q), \, {}^3D_4(q), \, G_2(q), \, {}^2G_2(q), \, {}^2B_2(q).$$
\end{prop}

\begin{proof}
In each of these cases, the maximal subgroups of $G$ have been determined; the relevant references are listed below (also see \cite[Chapter 4]{W_book}). Note that the list of maximal subgroups of ${}^2E_6(2)$ presented in the Atlas \cite{Atlas} is complete (see \cite[p.304]{Modat}).

\vs

\begin{center}
\begin{tabular}{lccccccc} \hline
$G$ & $E_{6}^{\e}(2)$ & $F_4(2)$ & ${}^2F_4(q)$ & ${}^3D_4(q)$ & $G_2(q)$ & ${}^2G_2(q)$ & ${}^2B_2(q)$ \\ \hline
\mbox{Ref.} & \cite{Atlas, KW} & \cite{NW} & \cite{Mal} & \cite{Kl3} & \cite{Coop, Kl2} & \cite{Kl2} & \cite{Suz} \\ \hline
\end{tabular}
\end{center}

\vs

Armed with this information, it is straightforward to verify Theorem \ref{t:main4} for these groups. In particular, we note that every maximal parabolic subgroup of $G$ is large.
\end{proof}

Let us now turn our attention to the remaining cases: 
$$G \in \{E_8(q), E_7(q),  E_6^{\e}(q),  F_4(q)\}$$
where $q=p^a$ and $p$ is a prime (and $q>2$ if $G=E_6^{\e}(q)$ or $F_4(q)$). Let $H$ be a maximal subgroup of $G$. If $H$ is a maximal parabolic subgroup then it is easy to check that $|H|^3 \geqs |G|$, so for the remainder, we will assume that $H$ is non-parabolic.

Let $\G$ be the ambient simple algebraic group defined over the algebraic closure $K$ of $\F$, and let $\s$ be a Frobenius morphism of $\G$ such that $G = (\G_{\s})'$. We will apply the following reduction theorem of Liebeck and Seitz (see \cite[Theorem 2]{LS10}).

\begin{thm}\label{t:ls}
Let $G = (\G_{\s})'$ be a finite simple group of exceptional Lie type, and let $H$ be a maximal non-parabolic subgroup of $G$. Then one of the following holds:
\begin{itemize}\addtolength{\itemsep}{0.2\baselineskip}
\item[(i)] $H=N_{G}(\M_{\s})$, where $\M$ is a $\s$-stable closed subgroup of $\G$ of positive dimension;
\item[(ii)] $H$ is an exotic local subgroup (as recorded in \cite[Table 1]{CLSS}); 
\item[(iii)] $G=E_8(q)$, $p>5$ and $H=(A_{5}\times A_{6}).2^2$;
\item[(iv)] $H$ is of the same type as $G$ over a subfield of $\F$ of prime index;
\item[(v)] $H$ is almost simple, and not of type (i) or (iv). 
\end{itemize}
\end{thm}

\begin{remk}\label{r:as}
Suppose $H$ is almost simple with socle $H_0$, as in part (v) of Theorem \ref{t:ls}. There are two possibilities:
\begin{itemize}\addtolength{\itemsep}{0.2\baselineskip}
\item[(i)] If $H_0 \not\in {\rm Lie}(p)$ then the possibilities for $H_0$ have been determined up to  isomorphism; see \cite[Tables 10.1--10.4]{LS3}.
\item[(ii)] Suppose $H_0 \in {\rm Lie}(p)$ is defined over $\mathbb{F}_{s}$ for some $p$-power $s$. By applying \cite[Theorem 1.1]{LS6} and \cite[Theorem 1]{LS11} we deduce that ${\rm rk}(H_0) \leqs \frac{1}{2}{\rm rk}(G)$ and one of the following holds (see \cite[Theorem 2]{Lawther} for the values of $u(G)$ in part (c)):
\begin{itemize}\addtolength{\itemsep}{0.2\baselineskip}
\item[(a)] $s \leqs 9$;
\item[(b)] $H_0 = A_2^{\e}(16)$;
\item[(c)] $H_0 \in \{A_1(s), {}^2B_2(s), {}^2G_2(s)\}$ and $s \leqs (2,p-1)\cdot u(G)$, where $u(G)$ is defined as follows: 
$$\begin{array}{c|ccccc}
G & G_2 & F_4 & E_6 & E_7 & E_8 \\ \hline
u(G) & 12 & 68 & 124 & 388 & 1312
\end{array}$$
\end{itemize}
\end{itemize}
\end{remk}

\begin{lem}\label{l:e8}
Let $H$ be a maximal non-parabolic subgroup of $G=E_8(q)$. Then $H$ is large if and only if $H$ is of type $A_1(q)E_7(q)$, $D_8(q)$, $A_2^{\e}(q)E_{6}^{\e}(q)$ or $E_8(q_0)$ with $q=q_0^2$.
\end{lem}

\begin{proof}
If $|H| >q^{88}$ then \cite[Lemma 4.2]{BLS} implies that $H$ is of type $E_8(q^{1/2})$, $A_1(q)E_7(q)$ or $D_8(q)$, and of course, $H$ is large in each of these cases. Clearly, if $|H| \leqs q^{82}$ then $H$ is non-large, so it remains to determine the maximal subgroups $H$ that satisfy the bounds $q^{82} < |H| \leqs q^{88}$. By Theorem \ref{t:ls}, $H$ is of type (i) -- (v).

First assume $H$ is of type (i). The possibilities for $\M$ are listed in \cite{LS2}, and it is easy to see that the only possibility with $|H|$ in the desired range is the case $\M = A_2E_6.2$. Here $$|H| \geqs |{\rm SL}_{3}(q)||E_6(q)|2>\frac{1}{2}q^{86}$$
(see \cite[Table 5.1]{LSS}) and thus $H$ is large. 

The possibilities in (ii) are recorded in \cite[Table 1]{CLSS}; either $|H|=2^{15}|{\rm SL}_{5}(2)|$, or $|H| = 5^3|{\rm SL}_{3}(5)|$ (both with $q$ odd). In both cases, $H$ is non-large. Clearly, we can eliminate subgroups of type (iii), and a straightforward calculation shows that a subfield subgroup $H=E_8(q_0)$ (with $q=q_0^k$, $k$ prime) is large if and only if $k=2$. 

Finally, let us assume $H$ is almost simple, and not of type (i) or (iv). Let $H_0$ denote the socle of $H$, and recall that ${\rm Lie}(p)$ is the set of finite simple groups of Lie type in characteristic $p$. First assume that $H_0 \not\in {\rm Lie}(p)$, in which case the  possibilities for $H_0$ are listed in \cite[Tables 10.1--10.4]{LS3}. If $H_0$ is an alternating or sporadic group then $|H| \leqs |{\rm Th}|$ and thus $H$ is non-large. Similarly, if $H$ is a group of Lie type then we get $|H| \leqs |{\rm PGL}_{4}(5)|2$, and again we deduce that $H$ is non-large. Finally, suppose $H_0 \in {\rm Lie}(p)$. By Remark \ref{r:as}(ii) we have ${\rm rk}(H_0) \leqs 4$, so $|H|<12q^{56}\log_pq$ by \cite[Theorem 1.2(i)]{LSh4}, and thus $H$ is non-large. 
\end{proof}

\begin{lem}\label{l:e7}
Let $H$ be a maximal non-parabolic subgroup of $G=E_7(q)$. Then $H$ is large if and only if $H$ is of type $(q-\e)E_{6}^{\e}(q)$, $A_1(q)D_6(q)$, $A_7^{\e}(q)$, $A_1(q)F_4(q)$ or $E_7(q_0)$ with $q=q_0^2$.
\end{lem}

\begin{proof}
We proceed as in the proof of the previous lemma. By \cite[Lemma 4.7]{BLS}, if $|H|>q^{46}$ then $H$ is of type $E_7(q^{1/2})$ or $N_G(\M_{\s})$, with $\M = T_1E_6.2$, $A_1D_6.2$, $A_7.2$ or $A_1F_4$. In each of these cases, $H$ is large. If $|H|<q^{44}$ then $|H|^3<|G|$, so to complete the analysis of this case we may assume that $q^{44} \leqs |H| \leqs q^{46}$. We now apply Theorem \ref{t:ls}.

By inspecting \cite{LS2} and \cite[Table 1]{CLSS}, it is easy to check that there are no examples of type (i) or (ii), and case (iii) does not arise. Next suppose $H$ is a subfield subgroup of type $E_7(q_0)$ with $q=q_0^k$. If $k \geqs 5$ then $|H|<q^{44}$, so let us assume $k=3$. If $q$ is even then 
$|G|=f(q)$ and $|H|=f(q_0)$, where
$$f(x) = x^{63}(x^2-1)(x^6-1)(x^8-1)(x^{10}-1)(x^{12}-1)(x^{14}-1)(x^{18}-1),$$
and we calculate that $|H|^3<|G|$. If $q$ is odd then $H$ is simple (note that a Cartan subgroup of order $(q_0-1)^7$ in ${\rm Inndiag}(E_7(q_0))$ does not lie in $G$), so $|G|=\frac{1}{2}f(q)$, $|H|=\frac{1}{2}f(q_0)$ and $H$ is non-large. 

To complete the analysis, let us assume $H$ is almost simple, and not of type (i) or (iv). Let $H_0$ denote the socle of $H$. If $H_0 \not\in {\rm Lie}(p)$ then by inspecting \cite[Tables 10.1--10.4]{LS3} we deduce that $|H| \leqs |{\rm Ru}|$, so $H$ is non-large. Finally, we may assume $H_0 \in {\rm Lie}(p)$ and ${\rm rk}(H_0) \leqs 3$ (see Remark \ref{r:as}(ii)). 
Here \cite[Theorem 1.2(ii)]{LSh4} states that $|H| < 4q^{30}\log_pq$, and thus $H$ is non-large.
\end{proof}

\begin{lem}\label{l:e6}
Let $H$ be a maximal non-parabolic subgroup of $G=E_6^{\e}(q)$, where $q>2$. Then $H$ is large if and only if $H$ is of type $(q-\e)D_{5}^{\e}(q)$ ($\e=-$), $A_1(q)A_5^{\e}(q)$, $F_4(q)$, $(q-\e)^2.D_4(q)$, $(q^2+\e q+1).{}^3D_4(q)$, $C_4(q)$ ($p \neq 2$), $E_6^{\pm}(q^{1/2})$ ($\e=+$), or $E_6^{\e}(q_0)$, $q=q_0^3$ and $q_0 \equiv \e \imod{3}$. 
\end{lem}

\begin{proof} 
If $|H| > q^{32}$ then \cite[Lemma 4.14]{BLS} implies that one of the following holds:
\begin{itemize}\addtolength{\itemsep}{0.2\baselineskip}
\item[(a)] $H=N_{G}(\M_{\s})$, where $\M= T_1D_5$, $A_{1}A_{5}$, $F_{4}$, $T_{2}D_{4}.S_{3}$ or $C_{4}$ ($p \neq 2$);
\item[(b)] $\e=+$ and $H$ is of type $E_{6}^{\pm}(q^{1/2})$.
\end{itemize}
%(Note that in (a) we exclude the case $\M=T_1D_5$ since $N_{G}(\M_{\s})$ is non-maximal.)
Clearly, $H$ is large in each of these cases (note that the maximality of $H$ implies that $\e=-$ if $H$ is of type $(q-\e)D_5^{\e}(q)$; see \cite[Table 5.1]{LSS}). Of course, if $|H| \leqs q^{25}$ then $H$ is non-large, so it remains to determine the possibilities for $H$ such that 
$q^{25}<|H| \leqs q^{32}$. As before, we consider the five cases arising in Theorem \ref{t:ls}.

Suppose $H$ is of type (i). Here $\M=A_2^3.S_3$ is the only possibility, in which case $|H| \leqs |A_2^{-}(q)|^36$ (see \cite[Table 5.1]{LSS}) and we deduce that 
$H$ is non-large.
If $H$ is of type (ii) then $H=3^6.{\rm SL}_{3}(3)$ (with $p \geqs 5$) is the only possibility, and $H$ is non-large. Case (iii) does not apply here. Now assume $H$ is a subfield subgroup of type $E_6^{\e}(q_0)$ with $q=q_0^k$ and $k \geqs 3$. If $k \geqs 5$ then $|H|^3<|G|$, so let us assume $k=3$. A straightforward calculation reveals that $H$ is non-large if $q_0 \not\equiv \e \imod{3}$. On the other hand, if $q_0 \equiv \e \imod{3}$ then $H={\rm Inndiag}(E_6^{\e}(q_0))$ (a Cartan subgroup of order $(q_0-1)^6$ in ${\rm Inndiag}(E_6^{\e}(q_0))$ is contained in $G$) and we deduce that $H$ is large. 

Finally, let us assume $H$ is almost simple and not of type (i) or (iv). Let $H_0$ denote the socle of $H$. First assume $H_0 \not\in {\rm Lie}(p)$. Here the possibilities for $H_0$ can be read off from \cite[Tables 10.1--10.4]{LS3}; it is straightforward to check that no large subgroups of this type arise. 
Therefore, to complete the proof of the lemma we may assume that $H_0 \in {\rm Lie}(p)$ and ${\rm rk}(H_0) \leqs 3$. 
Here \cite[Theorem 1.2(iii)]{LSh4} gives $|H| \leqs 4q^{28}\log_pq$, so some additional work is required. 
 
We proceed as in the proof of \cite[Theorem 1.2]{LSh4}, using the method described in \cite[Step 3, p.310]{LieSax}. Write $q=p^a$ and $H_0 = X_{r}^{\e'}(s)$, where $r={\rm rk}(H_0)$ and $s=p^b$. We consider the various possibilities for $X_r$ (with $r \leqs 3$) in turn. Recall that if $c \geqs 2$ and $d \geqs 3$ are integers (and $(c,d) \neq (2,6)$), then $c_d$ denotes the largest primitive prime divisor of $c^d-1$.

To illustrate the general approach, consider the case $H_0 = A_3(s)$ with $\e=+$. Here 
$$q^{25} < |H| \leqs |{\rm Aut}(H_0)|<s^{17}$$
and thus $b/a \geqs 25/17$. Now $p_{4b}$ divides $|H|$, and thus $|G|$, so $4b$ divides one of the numbers $6a,8a,9a,12a$, whence $b/a \in \{3,9/4,2,3/2\}$. Moreover, since $p_{3b}$ divides $|G|$ (note that $(p,b) \neq (2,2)$ since we are assuming that $q>2$) we deduce that $b/a \in \{3,2,3/2\}$. However, $H_0 \neq A_3(q^2)$ by the proof of \cite[Theorem 1.2]{LSh4}, and we have $|H|<q^{25}$ if $H_0 = A_3(q^{3/2})$, and $|H|>q^{32}$ if $H_0 = A_3(q^3)$. This eliminates the case $H_0 = A_3(s)$. 

The other cases are similar. For example, if $H_0 = A_2(s)$ and $\e=-$ then $b/a \geqs 5/2$ and $p_{3b}$ divides $|G|$, so $3b$ divides $18a,12a,10a$ or $8a$, and thus $b/a \in \{6,4,10/3,3,8/3\}$. Further, by considering $p_{2b}$, we deduce that $b/a \in \{6,4,3\}$. If $H_0=A_2(q^6)$ then $|H|>q^{32}$, and it is easy to check that $|H|^3<|G|$ if  $H_0 = A_2(q^3)$. Finally, the case $H_0=A_2(q^4)$ is eliminated in the proof of \cite[Theorem 1.2]{LSh4}. 

The case $H_0=G_2(s)$ requires special attention. Here $|H|<s^{16}$, so $b/a \geqs 25/16$ and by considering $p_{6b}$ we deduce that $H_0 = G_2(q^2)$ is the only possibility.  Moreover, as noted in Remark \ref{r:as}(ii), we may assume that $q=3$.
However, we claim that $G_2(9)$ is not a subgroup of $E_{6}^{\e}(3)$, so this case does not arise. To see this, suppose it is a subgroup and consider the composition factors in the restriction of $V_{27}$ to $G_2(9)$, where $V_{27}$ is one of the irreducible $27$-dimensional modules for $E_{6}^{\e}(3)$ over $\mathbb{F}_3$. Let $W$ be a non-trivial irreducible module for $G_2(9)$ over $\mathbb{F}_3$ with $\dim W \leqs 27$. Since $\dim W = 14$ is the only possibility, we quickly deduce that each involution in $G_2(9)$ acts on $V_{27}$ as $[-I_{8},I_{19}]$ (up to conjugacy). However, the involutions in $E_{6}^{\e}(3)$ act on $V_{27}$ as $[-I_{12},I_{15}]$ or $[-I_{16},I_{11}]$ (see \cite[Table 4]{LS3}), so we have reached a contradiction.
\end{proof}

\begin{lem}\label{l:f4}
Let $H$ be a maximal non-parabolic subgroup of $G=F_4(q)$, where $q>2$. Then $H$ is large if and only if $H$ is of type $B_{4}(q)$, $D_4(q)$, ${}^3D_4(q)$, $A_1(q)C_3(q)$ ($p \ne 2$), $C_4(q)$ ($p=2$), ${}^2F_4(q)$ ($p=2$, $\log_2q$ odd), $F_4(q_0)$ with $q=q_0^2$, or ${}^3D_4(2)$ and $q=3$.
\end{lem}

\begin{proof}
By \cite[Lemma 4.23]{BLS}, if $|H|>q^{22}$ then either $H=N_G(\M_{\s})$ with 
$$\M^0 \in \{ B_4, D_4, A_1C_3 \, (p \neq 2), C_4 \, (p=2)\},$$ 
or $H$ is of type $F_4(q^{1/2})$ or ${}^2F_4(q)$. Clearly, $H$ is large in each of these cases, and we also note that $H$ is non-large if $|H|<q^{17}$. Therefore, we may assume that $q^{17} \leqs |H| \leqs q^{22}$. As usual, we consider the possibilities for $H$ labelled (i) -- (v) in Theorem \ref{t:ls}, and we quickly reduce to case (v) (in particular, subfield subgroups of type $F_4(q^{1/3})$ are non-large, and maximality rules out the maximal rank subgroups of type $C_2(q)^2$ and $C_2(q^2)$ (with $p=2$)).

Suppose case (v) holds, so $H$ is almost simple, and not of type (i) or (iv). Let $H_0$ denote the socle of $H$. If $H_0 \not\in {\rm Lie}(p)$ then the possibilities for $H$ are recorded in \cite[Tables 10.1--10.4]{LS3}, and it is easy to check that no large examples arise if $q>3$. However, if $q=3$ then $H_0 = {}^3D_4(2)$ is a possibility. Indeed, the main theorem of \cite{LS3} implies that ${}^3D_4(2)$ is a \emph{Lie primitive} subgroup of the algebraic group $\bar{G}=F_4(\bar{\mathbb{F}}_{3})$ (that is, $H_0$ is not contained in a proper closed subgroup of positive dimension in $\bar{G}$). Now, the modular character table of $H_0$ (see \cite[p.251]{Modat}) shows that ${}^3D_4(2).3$ has a $52$-dimensional irreducible module $V$ over $\mathbb{F}_{3}$, which can be identified with the Lie algebra of $F_4$ (see \cite[p.489]{Norton}). We deduce that $H={}^3D_4(2).3$ is a large maximal subgroup of $G=F_4(3)$.
 
Now assume $H_0 \in {\rm Lie}(p)$ and $r={\rm rk}(H_0) \leqs 2$. Here \cite[Theorem 1.2(iv)]{LSh4} gives $|H| < 4q^{20}\log_pq$, so some additional work is required. There are several cases to consider, and we proceed as in the proof of the previous lemma.

Write $q=p^a$ and $H_0 = X_{r}^{\e}(s)$, where $s=p^b$. First assume $H_0 = A_2(s)$, so $q^{17} \leqs |H|<s^{10}$ and thus $b/a \geqs 17/10$. By considering the primitive prime divisor $p_{3b}$ of $|H|$ we deduce that $b/a \in \{4,2,8/3\}$. The case $b/a=4$ is ruled out in the proof of \cite[Theorem 1.2]{LSh4}, and Remark \ref{r:as}(ii) rules out the case $b/a=8/3$. Therefore $H_0 = A_{2}(q^2)$ is the only possibility, and we calculate that $|H|^3<|G|$ unless $q=4$ and $H={\rm Aut}(A_2(16))$. Similarly, if $H_0 = A_2^{-}(s)$ then $H$ is large if and only if $q=4$ and $H={\rm Aut}(A_2^{-}(16))$. However, we claim that $A_{2}^{\e}(16)$ is not a subgroup of $F_4(4)$. 

If $\e=-$, this follows from \cite[Lemma 4.5]{LSS}, so we may assume $\e=+$. Suppose $H_0=A_2(16)$ is a subgroup of $G=F_4(4)$ and consider the restriction to $H_0$ of the irreducible $26$-dimensional module $V_{26}$ for $G$. Let $x \in A_2(16)$ be an involution. With the aid of {\sc Magma} \cite{magma}, we can construct the irreducible modules for $H_0$ over $\mathbb{F}_4$, and we can calculate the Jordan form of $x$ on each of these modules. In this way, we deduce that $x$ has Jordan form $[J_2^8,J_{1}^{10}]$ on $V_{26}$ (where $J_i$ denotes a standard unipotent Jordan block of size $i$), but no involution in $F_4(4)$ acts on $V_{26}$ in this way (see \cite[Table 3]{Lawther}). This is a contradiction, and we conclude that $A_2(16)$ is not a subgroup of $F_4(4)$.

Next, let us assume $H_0 = C_2(s)$. Here $b/a \geqs 17/11$ since $|H|<s^{11}$, and by considering $p_{4b}$ we deduce that $b/a \in \{2,3\}$. The case $b/a=3$ is eliminated in the proof of \cite[Theorem 1.2]{LSh4}, so we can assume $H_0 = C_2(q^2)$.  
As noted in Remark \ref{r:as}(ii), such a subgroup is non-maximal if $q>3$, so let us assume $q=3$. We claim that $H_0=C_2(9)$ is not a subgroup of $G=F_4(3)$. Seeking a contradiction, suppose otherwise and let $V_{25}$ be the irreducible $25$-dimensional module for $G$. Using {\sc Magma} \cite{magma}, we can construct all the irreducible modules for $H_0$ over $\mathbb{F}_3$, and we deduce that there exists an involution $x \in H_0$ which acts on $V_{25}$ as $[-I_{8},I_{17}]$ (up to conjugacy). But no involutions in $F_4(3)$ act on $V_{25}$ in this way (see \cite[Table 4]{LS3}, for example), so we have reached 
a contradiction and thus $C_2(9)$ is not a subgroup of $F_4(3)$. The case $H_0 = B_2(s)$ is entirely similar. Finally, the remaining possibilities for $H_0$ are ruled out in the usual manner.
\end{proof}

\section{Almost simple groups}\label{s:almost}

In this section, we extend our study of large subgroups to almost simple groups. 
%In this section, we prove Theorem \ref{t:almost}. 
Let $G$ be an almost simple group with socle $G_0$,  
and let $H$ be a maximal subgroup of $G$ such that $G = HG_0$. Throughout this section, we will assume that $G \neq G_0$. 

\begin{prop}\label{p:almostan}
If $G_0 = A_n$ then $H$ is large only if one of the following holds:
\begin{itemize}\addtolength{\itemsep}{0.2\baselineskip}
\item[{\rm (i)}] $H \cap G_0$ is either intransitive or imprimitive on $\{1, \ldots, n\}$;
\item[{\rm (ii)}] $G=S_n$ and $(n,H)$ is one of the following:
$$\begin{array}{llll}
(5,{\rm AGL}_{1}(5)), & (6,{\rm PGL}_{2}(5)), & (7,{\rm AGL}_{1}(7)), & (8,{\rm PGL}_{2}(7)), \\
(9,{\rm AGL}_{2}(3)), & (10,A_6.2^2), & (12,{\rm PGL}_{2}(11)); & 
\end{array}$$
\item[{\rm (iii)}] $G=A_6.2 = {\rm PGL}_{2}(9)$ and $H = D_{20}$ or $[16]$;
\item[{\rm (iv)}] $G=A_6.2 = {\rm M}_{10}$ and $H = {\rm AGL}_{1}(5)$ or $[16]$;
\item[{\rm (v)}] $G = A_6.2^2$ and $H = {\rm AGL}_{1}(5) \times 2$ or $[32]$.
\end{itemize}
\end{prop}

\begin{proof}
This is an entirely straightforward calculation, arguing as in Section \ref{s:alt}. Note that the subgroup denoted by $[16]$ in parts (iii) and (iv) (and also $[32]$ in (v)) is a Sylow $2$-subgroup of $G$.
\end{proof}

Similarly, the following result for sporadic groups is easily obtained  by inspecting 
the complete list of maximal subgroups in \cite{WebAt} (recall that in this paper, we regard the Tits group ${}^2F_4(2)'$ as a sporadic simple group):

\begin{prop}\label{p:almostspor}
If $G_0$ is a sporadic simple group, then $H$ is large unless $(G,H)$ is one of the following:
$$\begin{array}{llll}
%({\rm J}_{2}.2,S_5), & 
({\rm Suz}.2, S_7), & ({\rm O'N}.2, 31{:}30), & ({\rm O'N}.2, A_6.2), & ({\rm O'N}.2, {\rm L}_{2}(7).2), \\
({\rm J}_{3}.2,19{:}18), & 
%({\rm He}.2, 5^2{:}4.S_4), & 
%({\rm HN}.2, 3^{1+4}{:}4.S_5), & 
({\rm Fi}_{24}, S_9 \times S_5), & ({\rm Fi}_{24}, {\rm L}_{2}(8){:}3 \times S_6), &
({\rm Fi}_{24}, S_7 \times 7{:}6), \\
({\rm Fi}_{24}, 7^{1+2}{:}(6 \times S_3).2), &
({\rm Fi}_{24}, 29{:}28), & ({}^2F_4(2), 13{:}12). & 
\end{array}$$
\end{prop}

Now let us assume $G$ is an almost simple group of Lie type. In the context of Theorem \ref{t:almost}, we are interested in the inequality
\begin{equation}\label{e:hg0}
|H \cap G_0|^3 \geqs |G_0|.
\end{equation}
In order to prove Theorem \ref{t:almost}, we may assume that $H$ is a \emph{novelty} subgroup of $G$, that is, $H \cap G_0$ is a non-maximal subgroup of $G_0$. In addition, we may assume that $H \cap G_0$ is non-parabolic (it is easy to check that $H$ is large if $H \cap G_0$ is a parabolic subgroup of $G_0$).

\begin{prop}\label{p:almostc}
The conclusion to Theorem \ref{t:almost} holds if $G_0$ is a classical group.
\end{prop}

\begin{proof}
As described in Section \ref{ss:substr}, Aschbacher's subgroup structure theorem \cite{asch} extends to all almost simple classical groups, with some suitable modifications. The eight geometric subgroup collections (denoted $\C_1, \ldots, \C_8$) can be defined as before, and once again we write $\mathcal{S}(G)$ for the additional family of almost simple irreducible subgroups. However, if $G$ contains certain automorphisms then we must consider some additional novelty subgroups that arise in the following three cases:
\begin{itemize}\addtolength{\itemsep}{0.2\baselineskip}
\item[{\rm (i)}] $G_0 = {\rm L}_{n}(q)$, $n \geqs 3$: If $G$ contains graph or graph-field automorphisms, then there is an extra family of geometric subgroups of $G$. This is denoted by 
$\C_1'$ in \cite[Section 13]{asch} (in \cite{BHR} and \cite{KL}, the $\C_1$ collection is redefined so as to include $\C_1'$). 
\item[{\rm (ii)}] $G_0 = {\rm PSp}_{4}(q)$, $q \geqs 4$ even: Here $G_0$ admits graph automorphisms, and the $\C_i$ families can be suitably modified in such a way that a version of Aschbacher's theorem still applies when $G$ contains such elements (see \cite[Section 14]{asch} and \cite[Table 8.14]{BHR}).
\item[{\rm (iii)}] $G_0 = {\rm P\O}_{8}^{+}(q)$: Aschbacher's theorem does not apply if $G$ contains triality automorphisms. Some partial information is given in \cite[Section 15]{asch}, and  a complete description of the maximal subgroups of $G$ was obtained by Kleidman \cite{K}. 
\end{itemize}

In each of these cases, it is straightforward to determine the possibilities for $H$ that satisfy the inequality in \eqref{e:hg0}. For instance, in (i) we may assume that $H$ is of type ${\rm GL}_{m}(q) \times {\rm GL}_{n-m}(q)$ (with $1 \leqs m < n/2$). Then \cite[Proposition 4.1.4]{KL} states that
$$|H \cap G_0| = d^{-1}|{\rm GL}_{m}(q)||{\rm SL}_{n-m}(q)|$$
with $d=(n,q-1)$, and we quickly deduce that \eqref{e:hg0} holds. Similarly, if $G_0 = {\rm PSp}_{4}(q)$ (with $q \geqs 4$ even) and $H$ is of type $O_{2}^{\e}(q) \wr S_2$ then $|H \cap G_0| = 8(q-\e)^2$ and we see that \eqref{e:hg0} holds if and only if $(q,\e) = (4,-)$. A convenient list of the cases that arise in (iii) is given in \cite[Table III]{K} (also see \cite[Table 8.50]{BHR}), and it is easy to read off the subgroups $H$ such that \eqref{e:hg0} holds. 

To complete the analysis of classical groups, it remains for us to consider the irreducible almost simple subgroups in $\mathcal{S}(G)$ (we may also assume that we are not in one of the cases labelled (i), (ii) and (iii) above). Let $n$ be the dimension of the natural module for $G_0$. If $n \leqs 12$ then we inspect the tables of maximal subgroups in \cite[Chapter 8]{BHR}, recalling that we may assume $H$ is a novelty subgroup. In this way, we find that there are precisely two cases that satisfy the bound in \eqref{e:hg0}, but which are not listed in Table \ref{tab:sg}:
$$(G_0, H \cap G_0) = ({\rm U}_{3}(5), {\rm L}_{2}(7)),\; (\O_{10}^{-}(2), {\rm M}_{12}).$$ 
Finally, suppose $n>12$. By applying Liebeck's upper bound on $|H|$ (see Theorem \ref{lieb}), we may assume that $n \leqs 22$, and then the desired result quickly follows from our earlier analysis in Section \ref{ss:nongeom}.
\end{proof}

\begin{prop}\label{p:almostex}
The conclusion to Theorem \ref{t:almost} holds if $G_0$ is an exceptional group.
\end{prop}

\begin{proof}
First let us assume $G_0$ is one of the groups in the statement of Proposition \ref{p:ex1}. In these cases, the maximal subgroups of $G$ have been completely determined, and we can immediately determine the subgroups $H$ that satisfy the bound in \eqref{e:hg0}. Here it is helpful to note that there are no novelty subgroups if 
$$G_0 \in \{{}^2F_4(q), {}^3D_4(q), G_2(q) \, (p \neq 3), {}^2G_2(q), {}^2B_2(q)\}.$$
We deduce that there are precisely four cases that satisfy the bound in \eqref{e:hg0}:
$$(G_0, H \cap G_0) = ({}^2E_6(2), {}^3D_4(2).3),\; ({}^2E_6(2), \O_7(3)),\; (F_4(2), S_6 \wr S_2),\; (F_4(2), {\rm Sp}_{4}(4).2).$$
Note that in the latter two cases, $H$ is of type $C_2(2)^2$ and $C_2(2^2)$, respectively (as recorded in Table \ref{tab:almost}).

In each of the remaining cases, Theorem \ref{t:ls} describes the structure of $H$, and the rest of the analysis in Section \ref{s:ex} goes through essentially unchanged. The only difference is that we now include some maximal rank subgroups that were excluded in our earlier analysis (since they are only maximal when $G$ contains certain automorphisms). These are the cases $G_0 = E_6(q)$ with $H$ of type $(q-1)D_5(q)$, and $G_0 = F_4(q)$ ($p=2$) with $H$ of type $C_2(q)^2$ or $C_2(q^2)$ (see \cite[Table 5.1]{LSS}).  The result follows.
\end{proof}

\vs

This completes the proof of Theorem \ref{t:almost}.

\section{Algebraic groups}\label{s:alg}

In this final section we prove Theorem \ref{t:main5}, extending our earlier work from finite to algebraic groups. Let $G$ be a linear algebraic group over an algebraically closed field, and let $A$ and $B$ be closed subgroups of $G$. As before, a factorisation of the form $G=ABA$ is called a \emph{triple factorisation} of $G$. 

\begin{prop}\label{p:ta}
Let $G$ be a connected linear algebraic group over an algebraically closed field. Let $A,B$  be closed subgroups of $G$ such that $G=ABA$. Then $\dim G \leqs 2\dim A+\dim B$. 
\end{prop}

\begin{proof}
Let $C=A \times B \times A$ and let $\phi:C \to G$ be the surjective morphism of affine varieties defined by $\phi: (x,y,z) \mapsto xyz$. Let $\{c_1, \ldots, c_r\}$ be a set of (right) coset representatives for the connected component $C^0$ of $C$. For each $i$, the closure $\overline{\phi(C^0c_i)}$ is irreducible and has dimension at most $\dim C^0$ (see \cite[Lemma 1.9.1(iii)]{Springer}, for example). Since $G = \bigcup_{i} \overline{\phi(C^0c_i)}$ is connected, and thus irreducible, it follows that $G=\overline{\phi(C^0c_i)}$ for some $i$, whence
$$\dim G = \dim \overline{\phi(C^0c_i)} \leqs \dim C^0 = 2\dim A+\dim B$$
as required.
\end{proof}

As an immediate corollary, if $G=ABA$ for closed subgroups $A$ and $B$, then 
$$\max\{3\dim A, 3\dim B\} \geqs \dim G.$$
This observation motivates the following definition, which is a natural algebraic group analogue of Definition \ref{d:large}.

\begin{defn}\label{d:large2}
Let $G$ be a linear algebraic group. A proper closed subgroup $H$ of $G$ is said to be \emph{large} if $3\dim H \geqs \dim G$.
\end{defn}

Our aim is to prove Theorem \ref{t:main5}, which describes all the large maximal closed subgroups of simple algebraic groups. We will consider the classical groups and exceptional groups separately. Note that our results do not depend on the isogeny type of $G$.

\subsection{Classical algebraic groups}

Let $G$ be a simple classical algebraic group over an algebraically closed field $K$ of characteristic $p \geqs 0$, with natural module $V$. Set $n=\dim V$. In \cite{LS4}, Liebeck and Seitz define six natural, or \emph{geometric}, subgroup collections, denoted $\C_i$, $1 \leqs i \leqs 6$. As a special case of their main theorem, \cite[Theorem 1]{LS4}, it follows that if $H$ is a maximal closed positive-dimensional subgroup of $G$ then either $H$ is contained in one of the $\C_i$ collections (with $i \neq 5$), or $H$ is almost simple (modulo scalars), $H^0 \not\cong G$ and $H^0$ acts irreducibly and tensor-indecomposably on $V$ (we write $\mathcal{S}$ for this additional collection of subgroups). This is a natural algebraic group analogue of Aschbacher's subgroup structure theorem \cite{asch} for finite classical groups. Note that if $H^0$ is a classical group then $V$ is not the natural module for $H^0$, nor its dual.

The structure of the geometric subgroups comprising the $\C_i$ collections is described in \cite[Section 1]{LS4}, and it is straightforward to compute the dimension of these subgroups, and then check whether or not the bound $3\dim H \geqs \dim G$ is satisfied. 
Note that the subgroups in $\C_5$ are finite, so this collection can be discarded.

The next result provides an algebraic group analogue of Theorem \ref{lieb}.

\begin{prop}\label{p:lie}
If $H \in \mathcal{S}$ is a maximal closed subgroup of $G$, then $\dim H \leqs 3n$.
\end{prop}

\begin{proof}
We may assume that $H^0$ is a simple algebraic group. Let $\rho:H^0 \to {\rm SL}(V)$ be the corresponding irreducible representation, where $\dim V = n$.

The result is clear if $H^0$ is an exceptional group. Indeed, the minimal dimension of a non-trivial irreducible $KH^0$-module is well-known (see \cite[Table 5.4.B]{KL}, for example) and we deduce that $n \geqs f(H^0)$, where $f$ is defined as follows:
$$\begin{array}{llllll} \hline
H^0 & E_8 & E_7 & E_6 & F_4 & G_2 \\ \hline
\dim H & 248 & 133 & 78 & 52 & 14 \\
f(H^0) & 248 & 56 & 27 & 25 & 6 \\ \hline
\end{array}$$

Now assume $H^0$ is a classical group, and recall that $V$ is not the natural module for $H^0$, nor its dual. Suppose $H^0 = {\rm SL}_{m}$ with $m \geqs 2$, so   $\dim H = m^2-1$. If $m \geqs 13$ then \cite[Theorem 5.1]{Lu} implies that $n \geqs m(m-1)/2$ and the desired bound follows. For $3 \leqs m \leqs 12$, the result follows by inspecting \cite[Tables A.6 -- A.15]{Lu}, and of course the case $m=2$ is trivial. Next suppose $H^0 = {\rm Sp}_{2m}$, with $m \geqs 2$. Then $\dim H = m(2m+1)$ and $n \geqs 2m^2-m-2$ if $m \geqs 6$ (see \cite[Theorem 5.1]{Lu} and \cite[Tables A.35--A.40]{Lu}), which is sufficient. Finally, if $2 \leqs m \leqs 5$ then the relevant tables in \cite{Lu} indicate that $n \geqs 2^m$ and the result follows. The remaining cases $H^0={\rm SO}_{2m+1}$ (with $m \geqs 3$) and $H^0={\rm SO}_{2m}$ ($m \geqs 4$) are entirely similar. Note that if $H^0=D_4$ then $n \geqs 26$ because the $8$-dimensional irreducible modules for $H^0$ afford isomorphisms $H^0 \cong D_4$.
\end{proof}

\begin{lem}\label{l:psla}
The conclusion to Theorem \ref{t:main5} holds if $G$ is a classical algebraic group. 
\end{lem}

\begin{proof}
As above, let $V$ be the natural $G$-module and set $n = \dim V$. It is straightforward to calculate the dimensions of the geometric subgroups of $G$ from their descriptions in  
\cite[Section 1]{LS4}, and so we can read off the large subgroups that arise. For example, if $G = {\rm SL}_{n}$ then the positive-dimensional subgroups in the various $\C_i$ collections are listed in Table \ref{tab:psl} (note that the $\C_3$ collection is empty, and we can ignore the finite subgroups comprising the $\C_5$ collection). It is easy to check that $H \in \C_i$ is large if and only if $i=1$ (in which case $H=P_k$ is a maximal parabolic subgroup, that is, $H$ is the stabiliser of a $k$-dimensional subspace of $V$), or $H$ is a $\C_2$-subgroup of type ${\rm GL}_{n/2} \wr S_2$, or $H$ is a $\C_6$-subgroup of type ${\rm Sp}_{n}$ or $O_n$.

For the remainder, let us assume $H \in \mathcal{S}$ is a maximal positive-dimensional subgroup of $G$. Let $\rho:H^0 \to {\rm SL}(V)$ be the corresponding irreducible representation. Since $\dim G \geqs n(n-1)/2$ and $\dim H \leqs 3n$ (by Proposition \ref{p:lie}), we may assume that $n \leqs 19$.

First assume $G = {\rm SL}_{n}$, so $\dim G = n^2-1$ and Proposition \ref{p:lie} implies that $n \leqs 9$. By maximality, $\rho(H^0)$ does not fix a nondegenerate form on $V$, so $V$ is not self-dual (as a $KH^0$-module). By inspecting \cite{Lu}, we deduce that $H^0 =A_2$ with $n=6$ is the only possibility, but $\dim H = 8$ and $\dim G = 35$, so this is not a large subgroup.

Next suppose $n \geqs 4$ is even and $G = {\rm Sp}_{n}$. Here $\dim G = n(n+1)/2$ so we may assume that $n \leqs 16$. Using \cite{Lu} we deduce that
$$H^0 \in \{A_1, A_2, A_3, A_4, B_2, B_3, B_4, C_2, C_3, C_4, G_2\}.$$
If $H^0=G_2$ then $\dim H = 14$, so $n \leqs 8$ and \cite[Table A.49]{Lu} implies that $(n,p)=(6,2)$ is the only possibility. Here $G_2<{\rm Sp}_{6}$ (with $p=2$) is a large subgroup, which we record in Table \ref{tab:alg_class}. If $H^0=B_2$ or $C_2$ then $\dim H = 10$ and $H$ is non-large since $n \geqs 10$ (see \cite[Table A.22]{Lu}). If $H^0=B_3$ or $C_3$ then $n=8$ is the only possibility (see \cite[Tables A.23 and A.32]{Lu}) and $\rho$ embeds $H^0$ in a group of type $D_4$, rather than type $C_4$. Similarly, if $H^0=B_4$ or $C_4$ then it is easy to check that $H$ is non-large. The cases with $H^0=A_k$ are quickly eliminated in a similar fashion. 

A similar argument applies if $G$ is an orthogonal group; the reader can readily verify that the only large maximal subgroups in $\mathcal{S}$ are as follows (see Table \ref{tab:alg_class}):
$$G_2<B_3 \, (p \neq 2),\, B_3<D_4,\, C_3<D_4 \, (p=2).$$
Note that $A_2.2<G_2<B_3$ if $p=3$.
\end{proof}

\begin{table}
$$\begin{array}{clll} \hline
\mbox{Collection} & \mbox{Type of $H$} & \mbox{Conditions} & \dim H \\ \hline
\C_1 & P_k & 1 \leqs k < n & n^2-1-k(n-k) \\
\C_2 & {\rm GL}_{n/t}\wr S_t & t \geqs 2 & n^2/t-1 \\
\C_4 & {\rm SL}_{a} \otimes {\rm SL}_{n/a} & 2 \leqs a < n/a & a^2+(n/a)^2-2 \\ 
& \bigotimes {\rm SL}_{a}^t & a,t \geqs 2,\,  n = a^t,\, (a,t) \neq (2,2) & (a^2-1)t \\
\C_6 & {\rm Sp}_{n} & \mbox{$n \geqs 4$ even} & n(n+1)/2 \\
& O_n & n \geqs 2,\, p \neq 2 & n(n-1)/2 \\ \hline
\end{array}$$
\caption{The maximal positive-dimensional geometric subgroups of ${\rm SL}_{n}$}
\label{tab:psl}
\end{table}

\subsection{Exceptional algebraic groups}

To complete the proof of Theorem \ref{t:main5}, let $G$ be a simple algebraic group of exceptional type $G_2$, $F_4$, $E_6$, $E_7$ or $E_8$ over an algebraically closed field $K$ of characteristic $p \geqs 0$. The main subgroup structure theorem for such groups is due to Liebeck and Seitz \cite{LS2}. Indeed, a complete list of the maximal closed subgroups of positive dimension in $G$ is given in \cite[Corollary 2(i)]{LS2}, and we can easily read off the large subgroups. For the reader's convenience we record the dimensions of the maximal parabolic subgroups in Table \ref{t:parab}, and we immediately deduce that all such subgroups are large. The large non-parabolic maximal subgroups are recorded in Table \ref{tab:alg_ex}.

\begin{table}
$$\begin{array}{r|rlllllll}
 & H=P_{1} & P_{2} & P_{3} & P_{4} & P_{5} & P_{6} & P_{7} & P_{8} \\ \hline
G=E_{8} & 170 & 156 & 150 & 142 & 144 & 151 & 165 &  191 \\
E_{7} & \hspace{1.7mm} 100 & 91 & 86 & 80 & 83 & 91 & 106 & \\
E_{6} & 62  & 57 & 53 & 49 & 53 & 62 & & \\
F_{4} & 37 & 32 & 32 & 37 & & & & \\
G_{2} & 9 & 9 & & & & & & \\
\end{array}$$
\caption{$G$ exceptional, $\dim P_i$}
\label{t:parab}
\end{table}

\appendix
\section{Large maximal subgroups of ${\rm GL}_{n}(q)$}\label{a:gln}

Let $G$ be the finite general linear group ${\rm GL}(V) = {\rm GL}_{n}(q)$. Triple factorisations $G=ABA$ are studied in \cite{ABP}, where $A$ and $B$ are maximal geometric subgroups that either stabilise a subspace of $V$ (so they are in the collection $\C_1$), or stabilise a decomposition $V=V_1 \oplus V_2$ with $\dim V_1 = \dim V_2$ (so they are 
$\C_2$-subgroups of type ${\rm GL}_{n/2}(q) \wr S_2$). With a view towards a more general  study of triple factorisations of ${\rm GL}_{n}(q)$, in this short appendix we determine the large maximal subgroups $H$ of ${\rm GL}_{n}(q)$. Of course, if ${\rm SL}_{n}(q) \leqs H$ then $H$ is large, and it is maximal if and only if it has prime index. Therefore, in the statement of our main result we will assume that ${\rm SL}_{n}(q) \not\leqs H$.

\begin{thm}\label{t:gln}
Let $H$ be a maximal subgroup of ${\rm GL}_{n}(q)$, where $n,q \geqs 2$ and ${\rm SL}_{n}(q) \not\leqs H$. Then $H$ is large if and only if $H$ is one of the subgroups recorded in Table \ref{tab:gln}.
\end{thm}

\begin{proof}
Let $G={\rm GL}_{n}(q)$ and let $Z \cong C_{q-1}$ be the centre of $G$. As before, we have $H \in \mathcal{C}(G) \cup \mathcal{S}(G)$ and we note that $H/Z$ belongs to the corresponding subgroup collection of $G/Z = {\rm PGL}_{n}(q)$. Moreover, the maximality of $H$ in $G$ implies that $H/Z<G/Z$ is maximal, so we can use the information in \cite[Chapter 4]{KL} and \cite[Section 8.2]{BHR} to determine the possibilities for $H$. For geometric subgroups $H \in \mathcal{C}(G)$ we proceed as before (see Section \ref{sss:linear}), and we omit the details. Note that the conditions listed in the final column of Table \ref{tab:gln} guarantee both the maximality and largeness of $H$. For example, $H$ is large if $n \geqs 3$, $q=q_0^2$ and $H$ is a $\C_8$-subgroup of type ${\rm GU}_{n}(q_0)$, and \cite[Proposition 4.8.5(I)]{KL} indicates that $H/Z<G/Z$ is maximal if and only if 
$$\frac{q-1}{[q_0+1,(q-1)/(q-1,n)]} = 1,$$
where $[a,b]$ denotes the lowest common multiple of $a$ and $b$. It is easy to check that this is equivalent to the condition $(n,q_0-1)=1$, which is recorded in Table \ref{tab:gln}.

Finally, if $H \in \mathcal{S}(G)$ then $|H| < (q-1)q^{3n}$ by Theorem \ref{lieb}, so we immediately deduce that $n \leqs 9$ and we can then inspect the relevant tables in \cite[Section 8.2]{BHR}. The result quickly follows. 
%Note that we include conditions in the final column of Table \ref{tab:gln} to guarantee the maximality of $H$. 
\end{proof}

\begin{table}
$$\begin{array}{cll} \hline
\mbox{Collection} & \mbox{Type of $H$} & \mbox{Conditions} \\ \hline
\C_1 & P_i & 1 \leqs i < n \\
\C_2 & {\rm GL}_{n/t}(q) \wr S_t & \mbox{See Remark \ref{r:c2max}} \\
\C_3 & {\rm GL}_{n/k}(q^k) & k \in \{2,3\} \\
\C_5 & {\rm GL}_{n}(q_0),\, q=q_0^k & \mbox{See Remark \ref{r:c5max}} \\ 
\C_6 & 3^2.{\rm Sp}_{2}(3) & \mbox{$n=3$ and $q \in \{4,7,13\}$} \\
& 2^{1+2}_{-}.O_{2}^{-}(2) & \mbox{$n=2$, $q \leqs 13789$ prime, $q \equiv \pm 3 \imod{8}$} \\ 
\C_8 & {\rm Sp}_{n}(q) & \mbox{$n \geqs 4$ even, $(n/2,q-1)=1$} \\
& O_{n}^{\e}(q) & \mbox{$n \geqs 3$, $q$ odd, $(n,q-1)=(n,2)$} \\
& {\rm GU}_{n}(q_0) & \mbox{$n \geqs 3$, $q=q_0^2$, $(n,q_0-1)=1$} \\
\mathcal{S} & 2 \times {\rm M}_{11} & (n,q) = (5,3) \\ 
& A_7 & (n,q) = (4,2) \\ 
& 10 \times {\rm L}_{2}(7) & (n,q) = (3,11) \\
\hline
\end{array}$$
\caption{Large maximal subgroups $H<{\rm GL}_{n}(q)$, ${\rm SL}_{n}(q) \not\leqs H$}
\label{tab:gln}
\end{table}

\begin{remk}\label{r:c2max}
Let $H = {\rm GL}_{n/t}(q) \wr S_t$ be a $\C_2$-subgroup. Then $H$ is large and maximal if and only if one of the following holds:
\begin{itemize}\addtolength{\itemsep}{0.2\baselineskip}
\item[{\rm (i)}] $t=2$ and $(n,q) \not\in \{(2,3),(2,5),(4,2)\}$;
\item[{\rm (ii)}] $t=3$, $q \geqs 3$ and $(n,q) \neq (3,4)$;
\item[{\rm (iii)}] $t=n=4$ and $q \in \{5,7\}$.
\end{itemize}
\end{remk}

\begin{remk}\label{r:c5max}
Let $H$ be a $\C_5$-subgroup of type ${\rm GL}_{n}(q_0)$, where $q=q_0^k$ and $k$ is a prime. Then 
$$|H|=|{\rm GL}_{n}(q_0)|\cdot \left(\frac{q-1}{q_0-1}\right)$$
and \cite[Proposition 4.5.3(I)]{KL} implies that $H$ is maximal if and only if
$$(n,(q-1)/(q_0-1))=1.$$
Therefore, $H$ is large and maximal if and only if one of the following holds:
\begin{itemize}\addtolength{\itemsep}{0.2\baselineskip}
\item[{\rm (i)}] $k \in \{2,3\}$ and $(n,(q-1)/(q_0-1))=1$; or
\item[{\rm (ii)}] $n=2$, $q \geqs 3$ and $k \in \{5,7\}$.
\end{itemize}
\end{remk}

\begin{remk}\label{r:tf}
Consider a triple factorisation of the form 
\begin{equation}\label{e:aba}
{\rm GL}_{n}(q) = ABA = BAB,
\end{equation}
where $A$ and $B$ are maximal subgroups, with $B \in \C_2$. Assume that $q$ is large ($q \geqs 13$ is sufficient). By Theorem \ref{t:gln}, we may assume that $B = {\rm GL}_{n/2}(q) \wr S_2$ or ${\rm GL}_{n/3}(q) \wr S_3$. In fact, the case $B={\rm GL}_{n/3}(q) \wr S_3$ does not arise.  Indeed, if $Z$ denotes the centre of ${\rm GL}_{n}(q)$ then $Z \leqs A \cap B$, so $\bar{G} = \bar{A}\bar{B}\bar{A} = \bar{B}\bar{A}\bar{B}$ (where $\bar{G} = G/Z$, etc.) and thus $|\bar{B}|^3 \geqs |\bar{G}|$. However, it is straightforward to check that if $B = {\rm GL}_{n/3}(q) \wr S_3$ then $|\bar{B}|^3 < |\bar{G}|$ for all $n \geqs 3$ and $q \geqs 13$. Triple factorisations as in \eqref{e:aba} are studied in \cite{ABP}, where $A \in \C_1$ and $B = {\rm GL}_{n/2}(q) \wr S_2$ is a $\C_2$-subgroup.
\end{remk}

\end{document}